\documentclass[11pt]{amsart}

\usepackage[USenglish]{babel}

\usepackage{amsmath,amsthm,amssymb,amscd}
\usepackage{booktabs}

\usepackage{MnSymbol}

\usepackage{tikz}
\usetikzlibrary{arrows,shapes.geometric,positioning,decorations.markings}

\usepackage[mathscr]{eucal}
\usepackage[normalem]{ulem}
\usepackage{latexsym,youngtab}
\usepackage{multirow}
\usepackage{epsfig}
\usepackage{parskip}

\usepackage[margin=2cm]{geometry}
\usepackage{color}

\newtheoremstyle{custom}% name
  {3pt}%      Space above
  {3pt}%      Space below
  {\slshape}%         Body font
  {}%         Indent amount (empty = no indent, \parindent = para indent)
  {\bfseries}% Thm head font
  {.}%        Punctuation after thm head
  { }%     Space after thm head: " " = normal interword space;
   {}%         Thm head spec (can be left empty, meaning `normal')
\theoremstyle{custom}
\newtheorem{theorem}{Theorem}[section]
\newtheorem{proposition}[theorem]{Proposition}
\newtheorem{proposition*}[theorem]{Proposition*}\newtheorem{theorem*}[theorem]{Theorem*}

\newtheorem{proposition/definition}[theorem]{Proposition/Definition}
\newtheorem{lemma}[theorem]{Lemma}
\newtheorem{corollary}[theorem]{Corollary}

\theoremstyle{definition}

\newtheorem{example}[theorem]{Example}

\theoremstyle{remark}
\newtheorem{remark}[theorem]{Remark}

\newtheoremstyle{exercise}% name
  {3pt}%      Space above
  {6pt}%      Space below
  {}%         Body font
  {}%         Indent amount (empty = no indent, \parindent = para indent)
  {\bfseries}% Thm head font
  {:}%        Punctuation after thm head
  { }%     Space after thm head: " " = normal interword space;
   {}%         Thm head spec (can be left empty, meaning `normal')
\theoremstyle{exercise}
\newtheorem{exercise}[theorem]{Exercise}
\newtheoremstyle{exercises}% name
  {3pt}%      Space above
  {6pt}%      Space below
  {}%         Body font
  {}%         Indent amount (empty = no indent, \parindent = para indent)
  {\bfseries}% Thm head font
  {:}%        Punctuation after thm head
  {\newline}%     Space after thm head: " " = normal interword space;
   {}%         Thm head spec (can be left empty, meaning `normal')

\theoremstyle{exercise}
\newtheorem{exercises}[theorem]{Exercises}

\input epsf
\def\boxit#1{\vbox{\hrule height1pt\hbox{\vrule width1pt\kern3pt
  \vbox{\kern3pt#1\kern3pt}\kern3pt\vrule width1pt}\hrule height1pt}}

\def\bv{\mathbf{v}}

\def\BC{\mathbb C}\def\BN{\mathbb N}

\def\BP{\mathbb P}
\def\pp#1{\mathbb P^{#1}}

\def\fb{\mathfrak b}

\def\pp#1{{\mathbb P}^{#1}}
\def\tdim{{\rm dim}}

\def\hd{,\ldots,}
\def\ww{\wedge}

\newcommand{\abs}[1]{\lvert#1\rvert}

\def\RR{\mathbb R}

\def\BB{{\mathbb B}}

\def\11{\mathbf 1}

\def\fsl{{\mathfrak {sl}}}

\def\fu{{\mathfrak u}}

\def\a{\alpha}

\def\o{\omega}

\def\s{\sigma}

\def\ot{{\mathord{ \otimes } }}
\def\op{{\mathord{\,\oplus }\,}}
\def\otc{{\mathord{\otimes\cdots\otimes} }}

\def\ra{{\mathord{\;\rightarrow\;}}}

\def\dim{{\rm dim}\;}
\def\La#1{\Lambda^{#1}}

\def\frak{\mathfrak}
\def\fsl{\frak s\frak l}

\def\op{\oplus}
\def\BZ{\Bbb Z}

\def\ep{\epsilon}
\def\op{\oplus}

\def\s{\sigma}
\def\t{\tau}

\def\a{\alpha}

\def\FS{\mathfrak  S}

\def\ol{\overline}

\def\BP{\mathbb  P}\def\BT{\mathbb  T}
\def\BC{\mathbb  C}

\def\pp#1{\mathbb  P^{#1}}

\def\tcodim{\text{codim}}

\def\ep{\epsilon}

\def\hd{, \hdots ,}

\def\La#1{\Lambda^{#1}}

\def\pp#1{\mathbb  P^{#1}}

\def\ur{\underline{\mathbf{R}}}

\def\ra{\rightarrow}

\def\tdet{\operatorname{det}}

\def\tdim{\operatorname{dim}}

\def\tlim{\lim}

\def\tmin{\operatorname{min}}

\def\ww{\wedge}

\def\bbb{{\mathbf{b}}}

\def\be{\begin{equation}}
\def\ene{\end{equation}}
\def\aaa{{\mathbf{a}}}
\def\bbb{{\mathbf{b}}}
\def\ccc{{\mathbf{c}}}
\def\tsgn{{\rm{sgn}}}

\newcommand{\Id}{\operatorname{Id}}

\def\Mn{M_{\langle \nnn \rangle}}\def\Mtwo{M_{\langle 2 \rangle}}\def\Mthree{M_{\langle 3\rangle}}

\def\Mn{M_{\langle \nnn \rangle}}\def\Mthree{M_{\langle 3\rangle}}

\def\Mtwo{M_{\langle 2\rangle}}\def\Mthree{M_{\langle 3\rangle}}

\def\aaa{\mathbf{a}}
\def\bbb{\mathbf{b}}
\def\ccc{\mathbf{c}}

\def\uuu{\mathbf{u}}

\def\mmm{\mathbf{m}}
\def\nnn{\mathbf{n}}
\def\lll{\mathbf{l}}

\def\bv{\mathbf{v}}
\def\bw{\mathbf{w}}

\def\Mn{M_{\langle \nnn \rangle}}\def\Mthree{M_{\langle 3\rangle}}

\def\Mtwo{M_{\langle 2\rangle}}\def\Mthree{M_{\langle 3\rangle}}

\def\aaa{{\bold a}} \def\ccc{{\bold c}}

\def\uuu{\bold u}

\def\mmm{\bold m}\def\nnn{\bold n}\def\lll{\bold l}

\def\bv{\bold v}\def\bw{\bold w}

\def\BB{\mathbb{B}}

\keywords{Matrix multiplication complexity, Tensor rank, Asymptotic rank, Laser method}


\renewcommand{\a}{\alpha}

\renewcommand{\BC}{\mathbb{C}}

\renewcommand{\hat}[1]{\widehat{#1}}

\begin{document}

\author[A. Conner, A. Harper, J.M. Landsberg]{Austin Conner, Alicia Harper, and  J. M. Landsberg}

\address{Department of Mathematics, Texas A\&M University, College Station, TX 77843-3368, USA}
\email[A. Conner]{connerad@math.tamu.edu}
\email[A. Harper]{adharper@math.tamu.edu}
 \email[J.M. Landsberg]{jml@math.tamu.edu}

\title[Border rank of matrix multiplication and $\tdet_3$]{New lower bounds for matrix multiplication and $\tdet_3$}

\thanks{Landsberg   supported by NSF grant  AF-1814254 and   the grant 346300 for IMPAN 
from the Simons Foundation and the matching 2015-2019 Polish MNiSW fund as well as a Simons Visiting Professor grant 
supplied by the Simons Foundation and by the Mathematisches Forschungsinstitut Oberwolfach.}



\begin{abstract}  
Let $M_{\langle \uuu,\bv,\bw\rangle}\in \BC^{\uuu\bv}\ot \BC^{\bv\bw}\ot \BC^{\bw\uuu}$ denote the matrix multiplication tensor
  (and write $\Mn=M_{\langle \nnn,\nnn,\nnn\rangle}$) and let $\tdet_3\in (\BC^9)^{\ot 3}$ denote
  the determinant polynomial considered as a tensor. For a tensor $T$, let $\ur(T)$ denote its border rank.
We  
(i) give the first hand-checkable algebraic proof that  $\ur(\Mtwo)=7$,
(ii) prove $\ur(M_{\langle 223\rangle})=10$,  and $\ur(M_{\langle 233\rangle})=14$, where previously  the only
nontrivial
matrix multiplication tensor  whose border rank had been determined was $\Mtwo$,
(iii) 
prove  $\ur( \Mthree)\geq 17$, 
(iv) 
prove   $\ur(\tdet_3)=17$, improving the  previous lower  bound of $12$,
 (v)  prove
$\ur(M_{\langle 2\nnn\nnn\rangle})\geq \nnn^2+1.32\nnn$ for all $\nnn\geq 25$ (previously
only $\ur(M_{\langle 2\nnn\nnn\rangle})\geq\nnn^2+1$ was known) as well as   lower bounds
for $4\leq \nnn\leq 25$, and (vi) prove
$\ur(M_{\langle 3\nnn\nnn\rangle})\geq \nnn^2+2\nnn$ for all  $\nnn\geq 21$, where
previously only $\ur(M_{\langle 3\nnn\nnn\rangle})\geq \nnn^2+2$ was known,
 as well as   lower bounds
for $4\leq \nnn\leq 21$,.

Our results utilize a new technique, called {\it border apolarity}
 developed by Buczy\'{n}ska and Buczy\'{n}ski in the general context
 of toric varieties. We apply this technique
 to tensors with symmetry to obtain
 an algorithm that, given a tensor $T$ with a large symmetry group and an integer $r$,   in a finite number of steps, either outputs that
there is no border rank $r$ decomposition for  $T$ or  produces a list of all potential
border rank $r$ decompositions in a natural normal form.  
The algorithm  is based on   algebraic geometry  and   representation
theory. 
The two key ingredients are: (i) the use of a multi-graded ideal associated to a   border rank $r$ decomposition of any tensor, and  
(ii)
 the exploitation of the large symmetry group of $T$
  to restrict to $\BB_T$-invariant ideals, where $\BB_T$ is a maximal solvable subgroup of the symmetry group of $T$.
 \end{abstract}

\maketitle

\section{Introduction}

Over fifty years ago  Strassen \cite{Strassen493} discovered that the usual row-column method for multiplying
$\nnn\times\nnn$ matrices,
which uses $O(\nnn^3)$ arithmetic operations, is not optimal by exhibiting an
explicit algorithm to multiply matrices using $O(\nnn^{2.81})$ arithmetic operations. Ever since then substantial efforts have
been made to determine just how efficiently matrices may be multiplied.
 See any of  \cite{BCS,blaserbook,MR3729273} for an overview.
 Matrix multiplication of $\nnn\times \lll$ matrices with $\lll\times \mmm$ matrices  is a bilinear map, i.e.,
a tensor $M_{\langle \lll,\mmm,\nnn\rangle}\in \BC^{\lll\mmm}\ot \BC^{\mmm\nnn}\ot \BC^{\nnn\lll}$, and   since 1980 \cite{MR605920}, the primary complexity
measure of the matrix multiplication tensor has been its {\it border rank} which is defined as follows:

A tensor $T\in \BC^\aaa\ot \BC^\bbb\ot \BC^\ccc=:A\ot B\ot C$ has {\it rank one} if   
$T=a\ot b\ot c$ for some $a\in A$, $b\in B$, $c\in C$, and
the   {\it rank} of  $T$, denoted $\bold R(T)$,  is the smallest $r$ 
such that $T$ may be written as a sum of $r$ rank one tensors. 
The {\it border rank} of $T$, denoted $\ur(T)$,   is the smallest $r$ such that 
$T$ may be written as a limit of a sum of $r$ rank one tensors. In geometric
language, the border rank is  smallest $r$ such that $[T]\in \sigma_r(Seg(\BP 
A\times \BP B\times \BP C))$.  Here $\sigma_r(Seg(\BP 
A\times \BP B\times \BP C))$  denotes  the $r$-th secant variety of the Segre variety of 
rank one tensors. 

Despite the vast literature on matrix multiplication, previous to this paper, the precise border rank
of $M_{\langle \lll,\mmm,\nnn\rangle}$ was known in exactly one nontrivial case, namely $\Mtwo=M_{\langle 222\rangle}$
\cite{MR2188132}. We determine the border rank in two new cases, $M_{\langle 223\rangle}$ and $M_{\langle 233\rangle}$.
We prove new border rank lower bounds for $\Mthree$ and two infinite series of new cases, $M_{\langle 2\nnn\nnn\rangle}$
and $M_{\langle 3\nnn\nnn\rangle}$. See \S\ref{results} below for precise statements.

The primary technique  for proving lower bounds previous to this work was using equations called {\it Koszul flattenings}
\cite{MR3081636,MR3376667} of which Strassen's equations \cite{Strassen505} is a special case. That is, these equations
are polynomials that vanish on all tensors of border rank less than a given $r$, and to prove
a tensor $T$  has border rank greater than $r$, one proves they do not
vanish on $T$. 
Koszul flattenings, even combined with
the {\it border substitution method}  \cite{MR3578455,MR3633766,MR3842382}  have been shown to be   near the limit 
of their utility for proving further  border
rank lower bounds \cite{MR3761737,MR3611482}. Our advances use  a new technique introduced in \cite{BBapolar}
in the larger context of the study of secant varieties of toric varieties, called {\it border apolarity}.
In a nutshell, border apolarity combines the classical apolarity method for Waring rank with the
border substitution method.

This is the first in a planned series of papers applying border apolarity to  
  tensors with symmetry.
 The technique is potentially
useful for both upper and lower bounds.
It is also potentially not subject to the known  barriers to proving lower bounds \cite{MR3761737,MR3611482} (see
\cite[\S 2.2]{LCBMS} for an overview), nor the known barriers to upper bounds
\cite{MR3388238,DBLP:conf/innovations/AlmanW18,2018arXiv181008671A,DBLP:journals/corr/abs-1812-06952}.
A border rank $r$ decomposition of a tensor $T$ is usually defined in terms of a curve
of tensors of rank $r$ limiting to $T$. The technique    replaces this   with more information:  a curve of
$\BZ^3$-graded ideals of smooth zero dimensional
schemes of length $r$ that limits to an ideal   with $T$ in its zero set.
The method then places numerous restrictions on ideals for them to arise 
from a border rank decomposition.  

For readers in computer science, we remark that despite the language used to derive
the algorithm, the border rank tests we derive from it for the matrix multiplication
tensor are elementary and easy to implement. Moreover, for any tensor with continuous  symmetry
one may obtain similar tests using only the rudiments of representation theory, namely
the definition of   weight vectors.

\subsection{Results}\label{results}

\begin{theorem}\label{mthreethm} $\ur(\Mthree)\geq 17$.
\end{theorem}

The previous lower bounds were $16$  \cite{MR3633766} in 2018,  $15$   \cite{MR3376667} in 2015, and   $14$  
\cite{Strassen505} in 1983.

Let $\tdet_3\in \BC^9\ot \BC^9\ot \BC^9$ denote the $3\times 3$ determinant polynomial considered as a tensor.

\begin{theorem} \label{detthreethm} $\ur(\tdet_3)= 17$.
\end{theorem}

The upper bound was proved in \cite{CGLVkron}.  In \cite{BOIJ2020195} a lower bound of $15$ for the Waring rank of $\tdet_3$
was proven.
The previous border rank lower bound was $12$ as discussed  in \cite{MR3427655},  which follows from the
Koszul flattening equations   \cite{MR3376667}.
Theorem \ref{detthreethm} has potential consequences for {\it upper} bounds on the exponent of matrix multiplication, 
see \cite{CGLVkron} for a discussion.

We   give   the first hand-checkable algebraic proof that $\ur(\Mtwo)=7$.
We actually give three such proofs. The first is the most elementary and only requires
the computation of the ranks of two sparse  size $40\times 24$ integer valued matrices.
The computations for the other two are even easier, while the proofs use 
some elementary representation theory. The previous algebraic proof \cite{MR3171099}  required the
evaluation of a degree $20$ polynomial in $64$ variables, which could only
be done with the help of a computer. The original differential-geometric
proof \cite{MR2188132} relied on a complicated case by case analysis of osculating spaces.

Previous to this paper $\Mtwo$ was the only
nontrivial matrix multiplication tensor whose border rank had been determined, despite $50$ years of work on the subject.
We add two   more cases to this list:

\begin{theorem} $\ur(M_{\langle 223\rangle})=10$.     \label{223thm}
\end{theorem}

The upper bound dates all the way back to Bini et. al. in 1980 \cite{MR592760}. 
Koszul flattenings  \cite{MR3376667} give $\ur(M_{\langle 22\nnn\rangle})\geq 3\nnn$.
Smirnov  \cite{MR3146566} showed that $\ur(M_{\langle 22\nnn\rangle})\leq 3\nnn+1$ for $\nnn\leq 7$,   
 and we expect equality to hold for all $\nnn$.

\begin{theorem}\label{2nnbnds} \

\begin{enumerate}
\item $\ur(M_{\langle 233\rangle})=14$.
\item  We have the following border rank lower bounds:
\begin{align*} \nnn \ \ \ \ \  & \ur(M_{\langle 2\nnn\nnn\rangle})\geq &\ \ \ \ \nnn \ \ \ \ \  & \ur(M_{\langle 2\nnn\nnn\rangle})\geq
&\ \ \ \ \nnn \ \ \ \ \  & \ur(M_{\langle 2\nnn\nnn\rangle})\geq \\
4 \ \ \ \ \  &22=4^2+6& \ \ \ 11 \ \ \ \ \  &136=11^2+15 & \ \ \  18 \ \ \ \ \  &348= 18^2+24    \\
5 \ \ \ \ \  &32=5^2+7 & \ \ \  12 \ \ \ \ \  &161=12^2+17 & \ \ \  19 \ \ \ \ \  &387=19^2+26     \\
6 \ \ \ \ \  &44=6^2+8& \ \ \  13 \ \ \ \ \  &187=13^2+18 & \ \ \  20 \ \ \ \ \  &427= 20^2+27    \\ 
7 \ \ \ \ \  &58=7^2+9& \ \ \  14 \ \ \ \ \  &216=14^2+20 & \ \ \  21 \ \ \ \ \  &470=  21^2+29   \\ 
8 \ \ \ \ \  &75=8^2+11& \ \ \  15 \ \ \ \ \  &246=15^2+21 & \ \ \  22 \ \ \ \ \  &514= 22^2+30    \\ 
9 \ \ \ \ \  &93=9^2+12& \ \ \  16 \ \ \ \ \  &278=16^2+22 & \ \ \  23 \ \ \ \ \  &561=  23^2+32   \\ 
10 \ \ \ \ \  &114=10^2+14& \ \ \  17 \ \ \ \ \  &312=17^2+23 & \ \ \  24 \ \ \ \ \  &609=  24^2+33 . 
%\ \  \ \ \ \ \  &\ \ \ \ \ & \ \ \  \ \  \ \ \ \ \  &\ \ \ \ \ \ \  & \ \ \  25 \ \ \ \ \  &658= 25^2+33      .
\end{align*}

\item For $0 < \epsilon < \frac{1}{4}$, and $\nnn>\frac{6}{\epsilon} \frac{3\sqrt{6}+6 - \epsilon}{6\sqrt{6} - \epsilon}$,
$\ur(M_{\langle 2\nnn\nnn\rangle})\geq \nnn^2+(3\sqrt{6}-6-\ep) \nnn+1$. In
    particular, $\ur(M_{\langle 2\nnn\nnn\rangle})\geq \nnn^2+1.32 \nnn+1$ when
    $\nnn\geq 25$.
\end{enumerate}
\end{theorem}

Previously it was only known that $\ur(M_{\langle 2\nnn\nnn\rangle})\geq \nnn^2+1$ by
\cite[Rem. p175]{MR86c:68040}.

The upper bound in (1) is due to Smirnov \cite{MR3146566}, where he also proved  $\ur(M_{\langle 244\rangle})\leq 24$, 
and  $  \ur(M_{\langle 255\rangle})\leq 38$.
  When $\nnn$ is even, one has the upper bound $\ur(M_{\langle 2\nnn\nnn\rangle})\leq \frac 74\nnn^2$
by writing $M_{\langle 2\nnn\nnn\rangle}=M_{\langle 222\rangle}\boxtimes M_{\langle 1\frac \nnn 2\frac \nnn 2\rangle}$, where $\boxtimes$
denotes Kronecker product of tensors, see, e.g., \cite{CGLVkron}. For general $\nnn$, we only know the trivial upper bound of $2\nnn^2$.

\begin{theorem}\label{mnnthm} \
\begin{enumerate}
\item For $4\leq \nnn\leq 13$, $\ur(M_{\langle 3\nnn\nnn\rangle})\geq   \nnn^2+2 \nnn -1$.
\item For $14\leq \nnn\leq 20$, $\ur(M_{\langle 3\nnn\nnn\rangle})\geq   \nnn^2+2 \nnn $.
\item For all $\nnn\geq 21$, $\ur(M_{\langle 3\nnn\nnn\rangle})\geq   \nnn^2+2 \nnn +1$.
\item For $0 < \epsilon < \frac{1}{2}$ and $ 
    \nnn > \frac{64}{3\epsilon} \frac{16\sqrt{78} - 96 - 21\epsilon}{32\sqrt{78}
    - 21\epsilon} $, 
    $\ur(M_{\langle 3\nnn\nnn\rangle})\geq \nnn^2+
    (\frac{16\sqrt{78}}{21}-\frac{32}7-\ep)\nnn + 1$.
\end{enumerate}
\end{theorem}

\begin{remark} We have slightly sharper results for $4\leq \nnn\leq 20$, see\newline
https://www.math.tamu.edu/$\sim$jml/bapolaritycode.html.
\end{remark}

Previously it was only known that when $\nnn\geq 4$, 
$\ur(M_{\langle 3\nnn\nnn\rangle})\geq \nnn^2+2$ by     \cite[Rem. p175]{MR86c:68040}
(see \S\ref{Lickappen}).

Using \cite[Rem. p175]{MR86c:68040}, one obtains

\begin{corollary} 
\label{3nncor} 
For all $\nnn\geq 21$, $\ur(M_{\langle \mmm\nnn\nnn\rangle})\geq   \nnn^2+2 \nnn +\mmm-2$.
\end{corollary}

\begin{remark} Koszul flattenings \cite{MR3376667} fail to give border rank lower bounds
for tensors in $A\ot B\ot C$ when the dimension of one of $A,B,C$ is much larger
than that of the other two. Theorems \ref{2nnbnds} and \ref{mnnthm}  show that the border apolarity method does not share this defect.
\end{remark} 

\subsection{Overview} 
In \S\ref{prelim} we review   terminology regarding border rank decompositions
of tensors and Borel fixed subspaces. We then  describe a  curve of multi-graded ideals one
may associate to a border rank decomposition. For   readers not familiar with representation
theory, we also review    Borel fixed (highest weight) subspaces.
In \S\ref{algsect} we describe
the border apolarity algorithm and accompanying tests. In \S\ref{mmultsect} we review
the matrix multiplication tensor. In \S\ref{mtwopf} we give our first hand-checkable
proof that $\ur(\Mtwo)=7$. In \S\ref{mthreepfsect} we describe the computation
to prove Theorems \ref{mthreethm} and  \ref{detthreethm}, which are   computer calculations, the code for
which is available at https://www.math.tamu.edu/$\sim$jml/bapolaritycode.html. In \S\ref{relrepsect} we discuss representation theory
relevant for applying the border apolarity algorithm to matrix multiplication, and use
it to get a shorter proof that $\ur(\Mtwo)=7$ and proofs of Theorems \ref{223thm} and \ref{2nnbnds}(1).
In \S\ref{outlinesect} we present the proofs of Theorems 
 \ref{2nnbnds}(3) and \ref{mnnthm}. 
In \S\ref{Lickappen}, for the convenience of the reader,  we give a proof of Lickteig's result that
$\ur(M_{\langle\lll,\mmm,\nnn\rangle})\geq \ur(M_{\langle\lll-1,\mmm,\nnn\rangle})+1$
for all $\lll,\mmm,\nnn$ which is used in the proof of Corollary \ref{3nncor}.

\subsection*{Acknowledgements} This project began when the third author  visited Institute of Mathematics of Polish Academy of Sciences (IMPAN), which he thanks for their hospitality,
  excellent environment for research, and support.  We are deeply indebted to Buczy\'{n}ska and Buczy\'{n}ski for sharing their preprint
with us and discussions. 
We thank Mateusz Micha{\l}ek for insight in settling Lemma \ref{lemma:opt}.

 \section{Preliminaries}\label{prelim}

\subsection{Definitions/Notation}

Throughout, $A,B,C,U,V,W$ will denote complex vector spaces respectively of dimensions $\aaa,\bbb,\ccc,\uuu,\bv,\bw$.
The dual space to $A$ is denoted $A^*$.
The identity map is denoted $\Id_A\in A\ot A^*$. For $X\subset A$, $X^\perp:=\{\a\in A^*\mid 
\a(x)=0\forall x\in X\}$ is its annihilator, and  $\langle X\rangle\subset A$ denotes the span of $X$.  
Projective space is  $\BP A= (A\backslash \{0\})/\BC^*$, and the Grassmannian of $r$ planes through
the origin is denoted $G(r,A)$, which we will view in its Pl\"ucker embedding $G(r,A)\subset \BP \La r A$.
The general linear group of invertible linear maps $A\ra A$ is denoted $GL(A)$
and the special linear group of determinant one linear maps is denoted $SL(A)$. 
 
For a set $Z\subset \BP A$, $\ol{Z}\subset \BP A$ denotes its Zariski closure,
$\hat Z\subset A$ denotes the cone over $Z$ union the origin, $I(Z)=I(\hat Z)\subset Sym(A^*)$ denotes
the ideal of $Z$,
and $\BC[\hat Z]=Sym(A^*)/I(Z)$, denotes the coordinate ring of $\hat Z$.
Both  $I(Z)$, $\BC[\hat Z]$ are $\BZ$-graded by degree.

One can define a $\BZ^3$-grading on the  ideals of subsets of $\BP A\times \BP B\times \BP C$.
%This is usually done by considering the sections of the line bundles
%$\cO_{\BP A}(i)\boxtimes\cO_{\BP B}(j)\boxtimes\cO_{\BP A}(k)$ but following
Here is an elementary way to define the grading, following
notes of Buczy\'{n}ski (personal communication):
Write ${\rm Irrel}:=\{ 0\op B\op C\} \cup \{ A\op 0\op C\}\cup \{ A\op B\op 0\}
\subset A\op B\op C$
and observe that
$\BP A\times \BP B\times \BP C=(A\op B\op C\backslash {\rm Irrel})/(\BC^*)^{\times 3}$.
Thus the closure of the  pullback    of any subset $Z\subset \BP A\times \BP B\times \BP C$
under the quotient map $q: (A\op B\op C)\backslash {\rm Irrel}\ra \BP A\times \BP B\times \BP C$
is invariant under the action of $(\BC^*)^{\times 3}$ and   $I(Z)=I(q^{-1}( Z))\subset
Sym(A\op B\op C)^*$
  is naturally $\BZ^3$-graded. 

Given $T\in A\ot B\ot C$, we may consider it as a linear map $C^*\ra A\ot B$, and we let $T(C^*)\subset A\ot B$
denote its image, and similarly for permuted statements.

\subsection{Border rank decompositions as curves in Grassmanians}\label{brasgrass}
A border rank $r$ decomposition of a tensor $T$  is normally viewed
as a curve $T(t)=\sum_{j=1}^r T_j(t)$ where each  $T_j(t)$ is rank one for all $t\neq 0$, and $\tlim_{t\ra 0}T(t)=T$.
It will be useful to change perspective, viewing a  border rank $r$ decomposition  of a tensor $T\in A\ot B\ot C$ as 
a curve $E_t\subset G(r,A\ot B\ot C)$ satisfying
\begin{enumerate}
\item for all $t\neq 0$, $E_t$ is spanned by $r$ rank one tensors, and 
\item $T\in E_0$.
\end{enumerate}

For example the border rank decomposition
$$a_1\ot b_1\ot c_2+ a_1\ot b_2\ot c_1+a_2\ot b_1\ot c_1=\tlim_{t\ra 0}\frac 1t[(a_1+ta_2)\ot (b_1+tb_2)\ot (c_1+tc_2) - a_1\ot b_1\ot c_1]
$$
may be rephrased as the curve
$$E_t=[(a_1\ot b_1\ot c_1)\ww (a_1+ta_2)\ot (b_1+tb_2)\ot (c_1+tc_2)]\subset G(2,A\ot B\ot C).
$$

Here 
$$
E_0=[(a_1\ot b_1\ot c_1)\ww (a_1\ot b_1\ot c_2+ a_1\ot b_2\ot c_1+a_2\ot b_1\ot c_1)].
$$

\subsection{Multi-graded ideal associated to a border rank decomposition}
Given a border rank $r$ decomposition $T=\tlim_{t\ra 0}\sum_{j=1}^r T_j(t)$, we have additional information:
Let 
$$I_t\subset Sym(A^*)\ot Sym(B^*)\ot Sym(C^*) 
$$
denote the $\BZ^3$-graded ideal of the set of  $r$ points $[T_1(t)]\sqcup\cdots
\sqcup [T_r(t)]$, where 
$I_{ijk,t}\subset S^iA^*\ot S^jB^*\ot S^kC^*$.
If the $r$ points are in general position, then   $\tcodim(I_{ijk,t})=r$ as long as
 $r\leq \tdim S^iA^*\ot S^jB^*\ot S^kC^*$
  (in our situation  $r$ will be sufficiently small so that  this will hold if at least two  of  $i,j,k$
are nonzero, see  e.g., \cite{MR3800460,MR3611482,2014arXiv1406.5145T}).
    For all $(ijk)$ with $i+j+k>1$, we may choose the curves such that  $\tcodim (I_{ijk})=r$   by \cite[Thm.~1.2]{BBapolar}.

Thus, in addition to $E_0=I_{111,0}^\perp$ defined in \S\ref{brasgrass}, we obtain a   limiting ideal $I$, where we define $I_{ijk}:=\tlim_{t\ra 0}I_{ijk,t}$ and the limit
is taken in the Grassmannian
$G(\tdim(S^iA^*\ot S^jB^*\ot S^kC^*)-r,S^iA^*\ot S^jB^*\ot S^kC^*)$.  
We remark that there are   subtleties here: the limiting ideal may  not be saturated. In particular,  the ideal of
the limiting scheme in the Hilbert scheme may  not agree with this limiting ideal. See \cite{BBapolar} for
a discussion. 
 
 Thus we may assume a  multi-graded ideal $I$ coming from    a border rank $r$ decomposition of a concise tensor $T$ satisfies the following     conditions:
\begin{enumerate}
\item[(i)] $I$ is contained in the annihilator of $T$. 
This condition says $I_{110}\subset T(C^*)^\perp$, $I_{101}\subset T(B^*)^\perp $, 
$I_{011}\subset T(A^*)^\perp$ and $I_{111}\subset T^\perp\subset A^*\ot B^*\ot C^*$.
\item[(ii)] For all $(ijk)$ with $i+j+k>1$, $\tcodim I_{ijk}=r$.  
\item[(iii)]  $I$ is an ideal, so the multiplication maps 
\be\label{ijkmap}
I_{i-1,j,k}\ot A^*\op I_{i,j-1,k}\ot B^* \op I_{i,j,k-1}\ot C^*\ra  S^iA^*\ot S^jB^*\ot S^kC^*
\ene
have image contained in $I_{ijk}$.
\end{enumerate}

One may  prove border rank lower bounds for $T$
by showing that for a given $r$,   no
such $I $ exists. For arbitrary tensors, we do not see any way to prove this, but for tensors with a 
nontrivial  symmetry group,
we have  a vast simplification of the problem as described in the next subsection.

\subsection{Lie's theorem and consequences} Lie's theorem may be stated as: 
Let $H$ be a  solvable group, let $W$ be an  $H$-module, and let $[w]\in \BP W$. Then
the orbit closure $\ol{H\cdot[ w]}$ contains an $H$-fixed point.

Assume $G_T$ is reductive (or contains a nontrivial reductive subgroup). Let $\BB_T\subset G_T$ be a maximal solvable subgroup, called a {\it Borel}
subgroup. By Lie's theorem and the Normal Form Lemma of \cite{MR3633766}, in order to prove $\ur(T)>r$, it is sufficient to disprove the existence
of a border rank decomposition   where $E_0$ is a $\BB_T$-fixed point of $\BP \La r(A\ot B\ot C)$.

By the same reasoning, as observed in \cite{BBapolar}, we may assume $I_{ijk}$ is $\BB_T$-fixed for all $i,j,k$. When $G_T$ is large, this can reduce
the problem to a finite, or nearly finite search.

Thus we may assume a  multi-graded ideal $I$ coming from    a border rank $r$ decomposition of $T$ satisfies the additional    condition:

\ \ \ \ (iv) Each $I_{ijk}$ is $\BB_T$-fixed.

As we explain in the next subsection, Borel fixed spaces are easy to list.

\subsection{Borel fixed  subspaces} For the reader's convenience, we review standard facts about Borel fixed subspaces.
In this paper only general and special  linear groups and   products of such appear. A Borel subgroup of $GL_m$ is just the group of invertible matrices
that are zero below the diagonal, and in products of general linear groups, the product of Borel subgroups is a Borel subgroup.
Let $\BC^m$ have basis $e_1\hd e_m$, with dual basis $e^1\hd e^m$. Assign $e_j$ weight $(0\hd 0,1,0\hd 0)$, where the $1$ is in 
the $j$-th slot and $e^j$ weight $(0\hd 0,-1,0\hd 0)$. 
  For vectors in $(\BC^m)^{\ot d}$,   $wt(e_1^{\ot a_1}\otc e_m^{\ot a_m})=(a_1\hd a_m)$  and the weight is unchanged under permutations of the $d=a_1+\cdots +a_m$ 
  factors.
  Partially order   the weights 
so that $(i_1\hd i_m)\geq (j_1\hd j_m)$ if $\sum_{\alpha=1}^si_{\alpha}\geq \sum_{\alpha=1}^sj_{\alpha}$ for all $s$.
The action of the Borel on a monomial $\mu$  sends it to a sum of monomials whose weights are   higher  than that
of $\mu$ in the partial order plus a monomial that is a scalar multiple of $\mu$.
Each irreducible $GL_m$ module appearing in the tensor algebra of $\BC^m$ has a unique highest weight which is given by a partition
$\pi=(p_1\hd p_m)$ and the module is denoted $S_{\pi}\BC^m$. Write $d=|\pi|=\sum p_i$. 
%Then
%$S_{\pi}\BC^m$ occurs in $(\BC^m)^{\ot d}$ with multiplicity $\tdim[\pi]$, where $[\pi]$
%is the irreducible module for the permutation group $\FS_d$ determined by $\pi$. 
See any of, e.g.,  \cite[\S 8.7]{MR3729273},
\cite[\S 9.1]{MR2265844}, or \cite[I.A]{MR1354144} for details. Let $\BT\subset GL_m$ denote the maximal torus of diagonal
matrices. A vector (or line) $w$  is a {\it    weight vector (line)}  
if the line $[w]$ is  fixed by the action of $\BT$.

We will  use $SL_m$ weights, which we write as $c_1\o_1+\cdots + c_{m-1}\o_{m-1}$, where the $\o_j$
are the fundamental weights. Here $wt(e_1)=\o_1$, $wt(e_m)=-\o_{m-1}$,  for $2\leq s\leq m-1$, 
$wt(e_s)=\o_s-\o_{s-1}$ and for all $j$,  $wt(e^j)=-wt(e_j)$. See the above references for explanations.

After fixing a (weight) basis of $\BC^m$,  an irreducible $G$-submodule $M$ of $(\BC^m)^{\ot d}$ has a basis of weight vectors, which
is unique up to scale  if $M$ is multiplicity free, i.e., there is at most one weight line
of  any given weight. In this case the $\BB$-fixed subspaces of dimension $k$, considered as elements of the
Grassmannian $G(k,M)$,  are just wedge products of  choices
of $k$-element subsets of the weight vectors  of $M$ such that no other element of $G(k,M)$, considered as a line in $\La kM$, has higher
weight in the partial order.
In the case a weight occurs with multiplicity in $M$, one has to introduce parameters in describing the subspaces.
In the case of direct sums of irreducible modules  $M_1\op M_2$, a subspace is $\BB$-fixed if it is 
spanned by weight vectors and, setting all the
$M_2$-vectors   in a basis of the subspace zero, what remains is a $\BB$-fixed subspace of $M_1$ and similarly
with  the roles of $M_1,M_2$ reversed.

In discussing weights, it is convenient to work with Lie algebras. Let $\fb$ denote the Lie algebra of $\BB$
and let $\fu\subset \fb$ be the space of   upper triangular matrices  with 
zero on the diagonal. We refer to elements of $\fu$ as {\it raising operators}. 
 A vector (or line) is a {\it  highest weight vector (line)}  
if it is a weight vector (line) annihilated by the action of   $\fu$.
A subspace of $M$ of dimension $k$  is $\BB$-fixed if and only if, considered as a line in 
$\La kM$, it is a highest weight line.

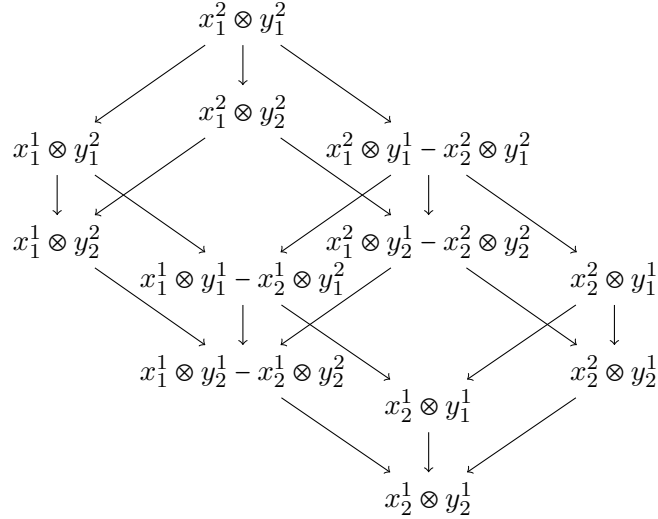
\begin{figure}[!htb]\begin{center}\label{uslvw222}
\begin{tikzpicture}
  \node (n000) at (0bp,0bp) {$x^2_1 \otimes y^2_1$};
  \node (n001) at (0bp,-35bp) {$x^2_1 \otimes y^2_2$};
  \node (n010) at (-70bp,-49bp) {$x^1_1 \otimes y^2_1$};
  \node (n011) at (-70bp,-84bp) {$x^1_1 \otimes y^2_2$};
  \node (n100) at (70bp,-49bp) {$x^2_1 \otimes y^1_1 - x^2_2 \otimes y^2_1$};
  \node (n101) at (70bp,-84bp) {$x^2_1 \otimes y^1_2 - x^2_2 \otimes y^2_2$};
  \node (n110) at (0bp,-98bp) {$x^1_1 \otimes y^1_1 - x^1_2 \otimes y^2_1$};
  \node (n111) at (0bp,-133bp) {$x^1_1 \otimes y^1_2 - x^1_2 \otimes y^2_2$};
  \node (n200) at (140bp,-98bp) {$x^2_2 \otimes y^1_1$};
  \node (n201) at (140bp,-133bp) {$x^2_2 \otimes y^1_2$};
  \node (n210) at (70bp,-146bp) {$x^1_2 \otimes y^1_1$};
  \node (n211) at (70bp,-181bp) {$x^1_2 \otimes y^1_2$};
  \draw [black,->] (n000) -> (n100);
  \draw [black,->] (n000) -> (n010);
  \draw [black,->] (n000) -> (n001);
  \draw [black,->] (n001) -> (n101);
  \draw [black,->] (n001) -> (n011);
  \draw [black,->] (n010) -> (n110);
  \draw [black,->] (n010) -> (n011);
  \draw [black,->] (n011) -> (n111);
  \draw [black,->] (n100) -> (n200);
  \draw [black,->] (n100) -> (n110);
  \draw [black,->] (n100) -> (n101);
  \draw [black,->] (n101) -> (n201);
  \draw [black,->] (n101) -> (n111);
  \draw [black,->] (n110) -> (n210);
  \draw [black,->] (n110) -> (n111);
  \draw [black,->] (n111) -> (n211);
  \draw [black,->] (n200) -> (n210);
  \draw [black,->] (n200) -> (n201);
  \draw [black,->] (n201) -> (n211);
  \draw [black,->] (n210) -> (n211);
\end{tikzpicture}
\caption{weight diagram for $U^*\ot \fsl(V)\ot W$ when $U=V=W=\BC^2$}
\end{center}
\end{figure}

\begin{example} When $U,V,W$ each have dimension $2$, Figure 1 gives  the weight diagram for 
$U^*\ot \fsl(V)\ot W$. Here, in each factor $\fu$ is spanned by the matrix
$\begin{pmatrix} 0&1\\0&0\end{pmatrix}$ and the labels on the edges indicate which
of the three raising  operators   acts to raise a weight (raising goes from
bottom to top). 
There is a unique $\BB$-fixed (highest weight) line, spanned by $x^2_1\ot y^2_1$,
(here $x^i_j=u^i\ot v_j$, $y^i_j=v^i\ot w_j$, and $z^i_j=w^i\ot u_j$)
three highest weight $2$-planes,  
$\langle x^2_1\ot y^2_1, x^1_1\ot y^2_1\rangle$, $\langle x^2_1\ot y^2_1, x^2_1\ot y^2_2\rangle$,
and $\langle x^2_1\ot y^2_1, x^2_1\ot y^1_1-x^2_2\ot y^2_1\rangle$,
four highest weight $3$-planes,
$\langle x^2_1\ot y^2_1, x^1_1\ot y^2_1, x^2_1\ot y^1_1-x^2_2\ot y^2_1\rangle$,
$\langle x^2_1\ot y^2_1,x^2_1\ot y^1_1-x^2_2\ot y^2_1, x^2_1\ot y^2_2\rangle$, 
$\langle x^2_1\ot y^2_1, x^1_1\ot y^2_1, x^2_1\ot y^2_2\rangle$, and 
$\langle x^2_1\ot y^2_1, x^2_1\ot y^1_1-x^2_2\ot y^2_1, x^2_2\ot y^1_1\rangle$,
etc..

\end{example}

\begin{figure}[!htb]\begin{center}\label{UotU2}
\begin{tikzpicture}
  \node (s1) at (0bp,0bp) {$u_1^2$};
  \node (s2) at (0bp,-30bp) {$u_1u_2$};
  \node (s3) at (40bp,-60bp) {$u_1u_3$};
  \node (s4) at (-40bp,-90bp) {$u_2^2$};
  \node (s5) at (0bp,-120bp) {$u_2u_3$};
  \node (s6) at (0bp,-150bp) {$u_3^2$};
  \draw [black,->] (s1) -> (s2);
  \draw [black,->] (s2) -> (s3);
  \draw [black,->] (s2) -> (s4);
  \draw [black,->] (s3) -> (s5);
  \draw [black,->] (s4) -> (s5);
  \draw [black,->] (s5) -> (s6);
  \node (a1) at (80bp,-30bp) {$u_1\wedge u_2$};
  \node (a2) at (80bp,-60bp) {$u_1\wedge u_3$};
  \node (a3) at (80bp,-120bp) {$u_2\wedge u_3$};
  \draw [black,->] (a1) -> (a2);
  \draw [black,->] (a2) -> (a3);
  \node (w1) at (140bp,0bp) {$2\omega_1$};
  \node (w2) at (140bp,-30bp) {$\omega_2$};
  \node (w3) at (140bp,-60bp) {$\omega_1-\omega_2$};
  \node (w4) at (140bp,-90bp) {$-2\omega_1+2\omega_2$};
  \node (w5) at (140bp,-120bp) {$-\omega_1$};
  \node (w6) at (140bp,-150bp) {$-2\omega_2$};
\end{tikzpicture}
\caption{weight diagram for $U\ot U$ when $U= \BC^3$.   There are $6$ distinct weights appearing, indicated on the right}
\end{center}
\end{figure}
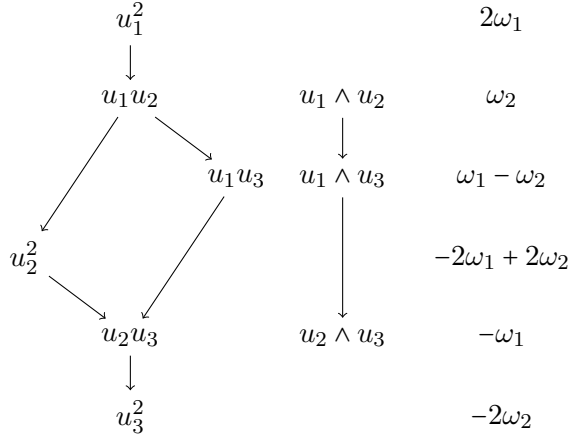

\begin{example}Let $\tdim U=3$. Figure 2 gives  the weight diagram for $U\ot U=S^2U\op \La 2 U$.
There are two $\BB$-fixed lines $\langle (u_1)^2\rangle$ and $\langle u_1\ww u_2\rangle$, there is a $1$-(projective) parameter $[s,t]\in\pp 1$ space of
$\BB$-fixed $2$-planes, $\langle (u_1)^2, su_1u_2+tu_1\ww u_2\rangle$ plus an isolated one $\langle u_1\ww u_2,u_1\ww u_3\rangle$ etc..
\end{example}

\section{The algorithm}\label{algsect}
{\bf Input}: An integer $r$ and a concise tensor $T\in A\ot B\ot C$ with symmetry group $G_T$ that contains a reductive
group (which by abuse of notation we denote $G_T$). 

{\bf Output}: Either a proof that $\ur(T)>r$ or a list of all Borel-fixed ideals that could potentially arise in a border rank
$r$ decomposition of $T$.

The following steps build an ideal $I$ in each multi-degree. We initially 
have $I_{100}=I_{010}=I_{001}=0$ (by conciseness), so the first spaces to
build are in total degree two. 

\begin{enumerate}

\item[(i)] 
%Decompose    $T(C^*)^{\perp}\subset A^*\ot B^*$ as a $G_T$
%module and draw the weight diagram of each irreducible component.
For each $\BB_T$-fixed  weight subspace $F_{110} $ of codimension $r-\ccc$ in $T(C^*)^{\perp }$ 
(and codimension $r$ in $A^*\ot B^*$)
compute the ranks of the maps 
\begin{align}
&\label{f210} F_{110}\ot A^*\ra S^2A^*\ot B^*, {\rm \  and}\\
&\label{f120} F_{110}\ot B^*\ra A^*\ot S^2B^*.
\end{align}
If both have images of codimension  at least  $r$, then   $F_{110}$  is a   candidate  
$I_{110}$. Call these maps the $(210)$ and $(120)$ maps and the rank conditions the
$(210)$ and $(120)$ tests.

\item[(ii)] Perform the analogous tests for potential $I_{101}\subset T(B^*)^\perp$
and $I_{011}\subset T(A^*)^{\perp}$ to obtain spaces $F_{101}$, $F_{011}$. 

\item[(iii)] For each triple $F_{110},F_{101},F_{011}$ passing the above tests, compute the rank of the map
\be\label{111map}
F_{110}\ot C^*\op F_{101}\ot B^*\op F_{011}\ot A^*\ra A^*\ot B^*\ot C^*.
\ene
If the codimension of the image is at least $r$, then one has a candidate  
triple. Call this map the $(111)$-map and the rank condition the $(111)$-test.
A space $F_{111}$ is a candidate for $I_{111}$ if it is of codimension $r$, contains the image of \eqref{111map}
and it is contained in $T^\perp$. 

\item[(iv)] 
%Decompose     $S^2A^*$ as a $G_T$
%module and draw the weight diagram of each irreducible component.
For each candidate triple $F_{110},F_{101},F_{011}$ obtained in the previous
step, and 
for each $\BB_T$-fixed subspace $F_{200}\subset S^2A^*$ of codimension $r$,
compute the rank of the maps $F_{110}\ot A^* \op F_{200}\ot B^*\ra S^2A^*\ot B^*$ and
$F_{101}\ot A^*\op F_{200}\ot C^* \ra S^2A^*\ot  B^*$. If the codimension of these images
is at least $r$, then one may add $F_{200}$ to the    candidate set.

Do the same for $\BB_T$-fixed subspaces $F_{020}$ and $F_{002}$, and collect all
total degree two candidate sets.

\item[(v)] Given an up until this point  candidate  set $\{F_{uvw}\}$ 
including   degrees   $(i-1,j,k)$,  $(i ,j-1,k)$, and $(i ,j,k-1)$,
%(ordered so
%$(ijk)\geq (i'j'k')$ if $i+j+k\geq i'+j'+k'$, $i+j\geq i'+j'$ and $i\geq i'$), 
compute the 
rank of the map
\be\label{ijkmap}
F_{i-1,j,k}\ot A^*\op F_{i,j-1,k}\ot B^*\ot F_{i,j,k-1}\ot C^*\ra S^iA^*\ot S^jB^*\ot S^kC^*.
\ene
If the codimension of the image of this map is less than $r$, the set is not a candidate.
Say  the codimension of the image is   $\xi\geq r$.
The image will be  $\BB_T$-fixed  by Schur's Lemma, as \eqref{ijkmap} is a $\BB_T$-module map. 
Each $(\xi-r)$-dimensional $\BB_T$-fixed subspace of the image (i.e.,  codimension $r$ $\BB_T$-fixed subspace of
$S^iA^*\ot S^jB^*\ot S^kC^*$ in the image) is a candidate
$F_{ijk}$.

\item[(vi)] If at any point there are no such candidates, we conclude $\ur(T)>r$. Otherwise, continue until stabilization occurs and one has a candidate ideal.
Repeat until one has all candidate ideals.  (Stabilization occurs at worst in multi-degree $(r,r,r)$,
see \cite{BBapolar}.)

\end{enumerate}

The output is either a certificate that $\ur(T)>r$ or  a collection  of   multi-graded ideals representing all possible
candidates for a $\BB_T$-fixed border rank decomposition. 
In current work with Buczynska et. al.   we are developing tests 
to determine if   a given  multi-graded ideal  comes from a border rank decomposition.

The algorithm above in total degree three   suffices to obtain the lower bounds proved in this article.

  One can perform the tests dually:
  
\begin{proposition}
The codimension of the image of the $(210)$-map is the dimension
of the kernel   
of the skew-symmetrization map 
\be\label{e210map}
F_{110}^\perp\ot A \ra \La 2 A\ot B.
\ene
The codimension of the image of the $(ijk)$-map  is the dimension of
\be\label{110inter} (F_{ij,k-1}^\perp\ot C)\cap (F_{i,j-1,k}^\perp\ot B)\cap(F_{i-1,j,k}^\perp\ot A).
\ene
\end{proposition}
\begin{proof}
The codimension of the image of the $(210)$-map is the dimension
of the kernel   
of its transpose,  $ S^2A\ot B\ra F_{110}^*\ot A$.
Since $A\ot A\ot B=A\ot (F_{110}^\perp \op F_{110}^*) =(S^2A\op \La 2 A)\ot B$,
the kernel equals $ (S^2A\ot B)\cap F_{110}^{\perp}$, which in turn 
is   the kernel of \eqref{e210map}.

The codimension of the image of the $(ijk)$-map is the dimension
of the kernel   
of its transpose,   
\begin{align*}
S^iA\ot S^jB\ot S^kC&\ra F_{ij,k-1}^*\ot C\op F_{i,j-1,k}^*\ot B\op F_{i-1,jk}^*\ot A\\
X&\mapsto {\rm Proj}_{F_{ij,k-1}^*\ot C}(X) \op  {\rm Proj}_{F_{i,j-1,k}^*\ot B}(X) \op 
 {\rm Proj}_{F_{i-1,jk}^*\ot A}(X)
\end{align*}
so $X$ is in the kernel if and only if all three projections are zero.
The kernels of the three projections are respectively 
$(F_{ij,k-1}^\perp\ot C)$, $(F_{i,j-1,k}^\perp\ot B)$, and $(F_{i-1.jk}^\perp\ot A)$, so we conclude.
\end{proof}

\section{Matrix multiplication}\label{mmultsect}
Let  $A=U^*\ot V$, $B=V^*\ot W$, $C=W^*\ot U$.
The matrix multiplication tensor
  $M_{\langle \uuu,\bv,\bw\rangle}\in A\ot B\ot C$ is the re-ordering of $\Id_U\ot \Id_V\ot \Id_W$.
Thus
  $G_{M_{\langle \uuu,\bv,\bw\rangle}}\supseteq GL(U)\times GL(V)\times GL(W)=:G$.
As a $G$-module $A^*\ot B^*= U\ot \fsl(V)\ot W^* \op U\ot \Id_V\ot W^*$.  
We have  $ M_{\langle \uuu,\bv,\bw\rangle}(C^*)= U^*\ot \Id_V\ot W$.
We fix bases and let $\BB$ denote the induced  Borel subgroup of   $G$.

For dimension reasons, it will be easier to describe $E_{ijk}:=F^{\perp}_{ijk}\subset S^iA\ot S^jB\ot S^kC$ than $F_{ijk}$.
Note that $E_{ijk}$ is $\BB$-fixed if and only if $E_{ijk}^\perp$ is.

Any candidate $E_{110}$ is an enlargement of $U^*\ot \Id_V\ot W$ obtained from choosing a $\BB$-fixed
$(r-\bw\uuu)$-plane inside $U^*\ot \fsl(V)\ot W$. Write $E_{110}=(U^*\ot \Id_V\ot W)\op E_{110}'$, where
$E_{110}'\subset U^*\ot \fsl(V)\ot W$ and $\tdim E_{110}'=r-\bw\uuu$.

We remark that since $\Mn$ has $\BZ_3$-symmetry (i.e., cyclic permutation of factors), to determine the candidate $I_{110}$, $I_{101}$ and $I_{011}$ it will suffice
to determine the candidate $I_{110}$'s.  Similarly, since $M_{\langle \nnn, \lll, \nnn\rangle}$
has $\BZ_2$-symmetry, the list of candidate $I_{110}$'s is isomorphic to the list of candidate $I_{011}$'s.

\section{$\Mtwo$}\label{mtwopf}
Here is a very short algebraic proof that $\ur(\Mtwo)=7$:

\begin{proof}
Here $\uuu=\bv=\bw=2$. We disprove border rank  at most six by showing no $\BB$-fixed
six dimensional $F_{110}$ (i.e., two dimensional $E_{110}'$)  passes both the $(210)$ and $(120)$ tests.
The weight diagram for $U^*\ot \fsl(V)\ot W$ appears in Figure 1.

There are only three $\BB$-fixed $2$-planes in $U^*\ot \fsl(V)\ot W$:
\begin{align*}
&\langle (u^2\ot v_1)\ot (v^2\ot w_1), (u^1\ot v_1)\ot (v^2\ot w_1)\rangle,\\
&\langle (u^2\ot v_1)\ot (v^2\ot w_1), (u^2\ot v_1)\ot (v^2\ot w_2)\rangle,\\
{\rm and\ }&\langle (u^2\ot v_1)\ot (v^2\ot w_1), (u^2\ot v_1)\ot (v^1\ot w_1)- (u^2\ot v_2)\ot (v^2\ot w_1)\rangle
\end{align*}
For the first, the rank of the
$40\times 24$ matrix of the map $E_{110}^\perp \ot A^* \ra S^2A^*\ot B^*$   is   $20>24-6=18$.
For the second, by symmetry,  the rank of the $(120)$-map is also $20$. For the third
 the rank of the $(210)$-map is $19$ and we conclude.
\end{proof}

\begin{remark} One can simplify the calculation of the rank of the map $E_{110}^\perp \ot A^* \ra S^2A^*\ot B^*$ by using
the map \eqref{e210map}.
 In the case above, the resulting
matrix is of size $24\times 24$.
The images of the basis vectors of $E_{110}\ot A$ in the case $E_{110}'= \langle x^2_1\ot y^2_1, x^1_1\ot y^2_1\rangle$ are
\begin{align*}
&x^1_1\ww x^2_1\ot y^2_1, x^1_2\ww x^2_1\ot y^2_1, x^2_2\ww x^2_1\ot y^2_1,\\
&  x^1_2\ww x^1_1\ot y^2_1, x^2_2\ww x^1_1\ot y^2_1,\\
&  x^1_1\ww (x^1_1\ot y^1_1+x^1_2\ot y^2_1)  , x^1_2\ww  (x^1_1\ot y^1_1+x^1_2\ot y^2_1),
 x^2_1\ww (x^1_1\ot y^1_1+x^1_2\ot y^2_1)  , x^2_2\ww  (x^1_1\ot y^1_1+x^1_2\ot y^2_1),\\
& x^1_1\ww  (x^2_1\ot y^1_1+x^2_1\ot y^2_1),x^1_2\ww  (x^2_1\ot y^1_1+x^2_1\ot y^2_1), 
 x^2_1\ww  (x^2_1\ot y^1_1+x^2_2\ot y^2_1),x^2_2\ww  (x^2_1\ot y^1_1+x^2_2\ot y^2_1),\\
&  x^1_1\ww (x^2_1\ot y^1_2+x^2_2\ot y^2_2)  , x^1_2\ww  (x^2_1\ot y^1_2+x^2_2\ot y^2_2), x^2_1\ww  (x^2_1\ot y^1_2+x^2_2\ot y^2_2),
x^2_2\ww  (x^2_1\ot y^1_2+x^2_2\ot y^2_2) \\
&  x^1_1\ww (x^2_1\ot y^1_2+x^2_2\ot y^2_2)  , x^1_2\ww  (x^2_1\ot y^1_2+x^2_2\ot y^2_2), x^2_1\ww  (x^2_1\ot y^1_2+x^2_2\ot y^2_2),
x^2_2\ww  (x^2_1\ot y^1_2+x^2_2\ot y^2_2) 
\end{align*}
and if we remove $x^2_1\ww  (x^2_1\ot y^1_1+x^2_2\ot y^2_1)$ we obtain a set of $20$ independent vectors.
In \S\ref{bvtwo} we give a further simplification of this calculation.
\end{remark}

\section{Explanation of the proofs  of Theorems \ref{mthreethm} and \ref{detthreethm}}\label{mthreepfsect}
The actual proofs to these theorems are in the code at the webpage 

https://www.math.tamu.edu/$\sim$jml/bapolaritycode.html.

What follows are explanations of what is carried out.

In the case of $\Mthree$,
  the weight zero subspace of  $\fsl_3$ has dimension two, so there are  
$\BB$-fixed  spaces  of dimension $7=16-4$
complementary to $U^*\ot \Id_V\ot W$ in $A\ot B$ that arise in positive dimensional
families. Fortunately the  set of  $7$-planes that   pass the
$(210)$ and $(120)$ tests is finite. 
When there are no parameters present, these tests consist of 
computing the ranks of the
$144\times 405$ matrices of $E_{110}^\perp \ot A^*\ra S^2A^*\ot B^*$, and 
$E_{110}^\perp \ot B^*\ra  A^*\ot S^2B^*$. When there are parameters, one determines   the ideal
in which the rank drops to the desired value.
There are eight $7$-planes that do pass the test, giving rise to $512$ possible triples (or $176$ triples
taking symmetries into account). Among the candidate triples, none pass  the $(111)$-test.

We now describe the relevant module structure for the determinant:
Write $U,V=\BC^m$ and $A_1=\cdots =A_m=U\ot V$. The determinant  $\tdet_m$, considered as
a tensor, spans  the line $\La m U\ot \La m V\subset A_1\otc A_m$.
Explicitly, letting $A_\alpha$ have basis $x^{\alpha}_{ij}$,   
$$\tdet_m=\sum_{\s,\t\in \FS_m}\tsgn(\t) x^1_{\s(1)\t(1)}\otc x^m_{\s(m)\t(m)}= \sum_{\s,\t\in \FS_m}\tsgn(\s) x^1_{\s(1)\t(1)}\otc x^m_{\s(m)\t(m)}.
$$

We will be concerned with the case $m=3$, and we write $A_1\ot A_2\ot A_3=A\ot B\ot C$. As a tensor, 
$\tdet_3$ is invariant under $(SL(U)\times SL(V))\rtimes \BZ_2$ as well as $\FS_3$.
In particular, to determine the candidate $E_{110}$'s it is sufficient to look in $A\ot B$,
which, as an $SL(U)\times SL(V)$-module is $U^{\ot 2}\ot V^{\ot 2}
=S^2U\ot S^2V \op S^2U\ot \La 2 V\op \La 2U\ot S^2 V \op \La 2 U\ot \La 2 V$, and $\tdet_3(C^*)=\La 2 U\ot \La 2 V$.

In the case of $\tdet_3$, 
 each of the three modules in  the complement to $\tdet_3(C^*)$ in $A\ot B$ are multiplicity free, 
 but there are weight multiplicities up to three, e.g.,
$\langle u_1u_2\ot v_1v_2, u_1u_2\ot v_1\ww v_2,  u_1\ww u_2\ot v_1v_2\rangle$ all have weight $(110|110)$.
We   examine all $7$-dimensional $\BB$-fixed subspaces of $S^2U\ot S^2V \op S^2U\ot \La 2 V\op \La 2U\ot S^2 V$,
which occur in positive dimensional families.
There are four candidates passing the $(210)$ and $(120)$ tests, but no triples passed the $(111)$ test.

In both cases, 
for the $E_{110}'$ with parameters, to perform the test we first perform  row reduction by constant entries. This usually reduces the problem
enough to take minors even with parameters. If it does not, 
we use the following algorithm, which effectively allows us to do
row reduction:
  First,  generalize to matrix entries in some quotient of
some ring of fractions of the polynomial ring, say $R$.
If there is a matrix entry which is a unit, pivot by it, reducing the
problem.  
Otherwise, select a nonzero entry, say $p$.  
Recursively compute the target ideal in two cases:
1. Pass to $R/(p)$, the computation here is smaller because the entry is zeroed.
2. Pass to $R_p$, the computation here is smaller because now $p$ is a
unit, and one can pivot by it.
Finally lift the ideals obtained by 1 and 2 back to $R$, say to $J_1$ and $J_2$,
and take $J_1J_2$.  Its zero set is the rank $<r$ locus and computing with it is tractable.

\section{Representation theory relevant for matrix multiplication}\label{relrepsect}
 Theorems \ref{223thm} and \ref{2nnbnds}(1),(2)  may also be proved using computer calculations but we present hand-checkable proofs to
both illustrate the power of the method and lay groundwork for future results. This section 
establishes the representation theory needed  for those
proofs. 
  
We have the following decompositions as $SL(U)\times SL(V)$-modules:
(note $V_{\o_2+\o_{\bv-1}}$ does not appear when $\bv=2$, and when $\bv=3$, $V_{\o_2+\o_{\bv-1}}=V_{2\o_2}$):
\begin{align}
\La 2(U^*\ot V)\ot V^*&=(S^2U^*\ot V)\op 
(\La 2  U^*\ot V)\op 
(S^2U^*\ot V_{\o_2+\o_{\bv-1}})\op 
(\La 2U^*\ot V_{2\o_1+\o_{\bv-1}})\\ 
S^2(U^*\ot V)\ot V^*&=(S^2U^*\ot V_{2\o_1+\o_{\bv-1}})\op 
(\La 2  U^*\ot V_{\o_2+\o_{\bv-1}})\op 
(S^2U^*\ot V)\op 
(\La 2U^*\ot V )\\ 
A\ot T(C^*)&=(U^*\ot V)\ot (U^*\ot \Id_V\ot W) = (S^2U^*\ot V\ot W)\op 
(\La 2U^*\ot V \ot W)\\ 
\label{vslv} V\ot \fsl(V)&=V_{2\o_1+\o_{\bv-1}}\op V_{\o_2+\o_{\bv-1}}\op V,\\
\nonumber &\ \\
\label{twelve}
(U^*\ot V)\ot (U^*\ot \fsl(V))&=
(S^2U^*\ot V_{2\o_1+\o_{\bv-1}})\op 
(\La 2  U^*\ot V_{2\o_1+\o_{\bv-1}}) \\
&\nonumber
\op(S^2U^*\ot V )\op 
(\La 2U^*\ot V )\\
&\nonumber
\op
(S^2U^*\ot V_{\o_2+\o_{\bv-1}})\op 
(\La 2U^*\ot V_{\o_2+\o_{\bv-1}}).
\end{align}
Note that 
$$\textstyle \tdim(V_{2\o_1+\o_{\bv-1}}) = \frac {1}{2} \bv^3+ \frac{1} {2} \bv^2 -\bv , \ \ \ \tdim(
  V_{\o_2+\o_{\bv-1}} )=\frac {1}{2}\bv^3-\frac{1}{ 2}\bv^2-\bv.
$$

The map $(U^*\ot V)\ot (U^*\ot \Id_V\ot W)\ra \La 2 (U^*\ot V)\ot (V^*\ot W)$ is injective, which implies:

\begin{proposition}\label{210kerprop}
Write $E_{110}:=M_{\langle \uuu,\bv,\bw\rangle}(C^*)\op E_{110}'$. The dimension of the kernel
of the  map \eqref{e210map} $E_{110}\ot A \ra \La 2 A\ot B$
equals the dimension of the kernel of 
the map
\be\label{210map}
E_{110}'\ot A \ra S^2U^*\ot V_{\o_2+\o_{\bv-1}}\ot W\op \La 2U^*\ot V_{2\o_1+\o_{\bv-1}}\ot W.
\ene
In other words, the dimension of the kernel is the dimension of
$$
(E_{110}'\ot A )\cap [ (U^*)^{\ot 2}\ot V\ot W\op S^2U^*\ot V_{2\o_1+\o_{\bv-1}}\ot W\op \La 2 U^*\ot V_{\o_2+\o_{\bv-1}}\ot W].
$$
\end{proposition} 

The \lq\lq in other words\rq\rq\ assertion follows by applying Schur's lemma using  \eqref{twelve} as 
the map  \eqref{210map} is equivariant.

\begin{remark}\label{zerorem}
The included module $V \subset V\ot \fsl(V) $ has basis
$\sum_{j\neq i} [v_j\ot (v_i\ot v^j)-  v_i\ot (v_j\ot v^j)]  -\frac 1{\bv-1}  v_i\ot v_i\ot v^i$.
To see this set $i=1$ and note that the vector is contained in $V\ot \fsl(V)$
and is annihilated by raising operators. 
In particular, the $i$-th basis vector is a linear combination of all monomials
in $ V\ot \fsl(V) $
of the same weight as $v_i$. We conclude that there is no contribution to the module 
$U^*\ot U^*\ot V\ot W$ in the kernel of the $(210)$ map of  $V$-weight
$v_i$ unless all vectors of a   weight $v_i$ appear
in $A\ot E_{110}'$.
\end{remark}

\begin{remark}\label{uwsymrem} 
In what follows, when $\tdim U=\tdim W=\nnn$, we will utilize the  $U^*\leftrightarrow W$ symmetry,
which also exchanges $V$ and $V^*$ to restrict attention to the $(210)$-map and to
impose symmetry.
\end{remark} 

\subsection{Case $\uuu=\bv=2$ and  $\Mtwo$ revisited}
Consider the special case $\uuu=\bv=2$:
We have the following images in $\La 2U^*\ot S^2V\ot V^*\ot W$:

For the highest weight vector $x^2_1\ot y^2_1$ times the four basis vectors of $A$
(with their $\fsl(V)$-weights in the second column, where we suppress the $\o_1$ from the notation),   the image is spanned by
\begin{align*}
x^1_1\ww x^2_1\ot y^2_1 &  \ \ 3\\
x^1_2\ww x^2_1\ot y^2_1 &  \ \ 1
\end{align*}
(Note, e.g., $x^2_2\ot x^2_1\ot y^2_1$ maps to zero as $u^2\ot u^2$ projects to
zero in $\La 2 U^*$.)
For  $x^2_1\ot y^1_1-x^2_2\ot y^2_1$ (the  lowering of $x^2_1\ot y^2_1$ under $\fsl(V)$), the image is spanned by 
\begin{align*}
x^1_1\ww (x^2_1\ot y^1_1-x^2_2\ot y^2_1) &  \ \ 1\\
x^1_2\ww (x^2_1\ot y^1_1-x^2_2\ot y^2_1) &  \ \ -1
\end{align*}
Observe that since $W$ has nothing to do with the map, we don't need
to compute the image of, e.g., $A\ot x^2_1\ot y^2_2$ to know its contribution,
as it must be the same dimension as that of $A\ot x^2_1\ot y^2_1$, just with a different $W$-weight.

These remarks already give an even shorter proof that $\ur(\Mtwo)=7$. Were it
$6$, $E_{110}'$ would have dimension two, it would be spanned by
the highest weight vector and one lowering of it, and in order to be a candidate,
its image in $\La 2U^*\ot S^3V\ot W$ would have to have dimension at
most two.
Taking $E_{110}'=\langle 
x^2_1\ot y^2_1, x^2_1\ot y^1_1-x^2_2\ot y^2_1\rangle$, the image
of \eqref{210map}  has dimension three.
Taking 
$E_{110}'=\langle 
x^2_1\ot y^2_1, x^2_1\ot y^2_2\rangle$, the image of \eqref{210map}  has dimension four.
Finally, taking
$E_{110}'=\langle 
x^1_1\ot y^2_1, x^2_1\ot y^2_1\rangle$, by symmetry (swapping the roles
of $U^*$ and $W$, which corresponds to taking transpose), the image of the
$(120)$-version of \eqref{210map}  must have dimension four and we conclude.

\subsection{Proof of Theorem \ref{223thm}} Here we break symmetry, taking  $\uuu=2$, $\bw= 3$, $\bv=2$.
We   show that there is no $E_{110}'$ of dimension $3=9-6$ passing the
$(210)$ and $(120)$ tests.
  There are eight  $\BB$-fixed $3$-planes in  $U^*\ot \fsl(V)\ot W$:
  three with just lowerings in $\fsl(V)$, four with one lowering in $\fsl(V)$ and one in either $U^*$ or $W$,
and one with no lowering in $\fsl(V)$  and one lowering each in $U^*,W$.
The only one passing the $(210)$-test is 
$\langle x^2_1\ot y^2_1, x^1_1\ot y^2_1, x^2_1\ot y^1_1-x^2_2\ot y^2_1\rangle$
where the kernel is $9$-dimensional. But for the $(120)$-test, the kernel is
$7$-dimensional and we conclude.
\qed

\subsection{Proof of   Theorem \ref{2nnbnds}(1),(2)}\label{pf2nnsect}
For Theorem \ref{2nnbnds}(1), 
   $\uuu=\bw=3$, $\bv=2$,  and $r=13$.  There are nine  $\BB$-fixed   four-dimensional 
subspaces of $U^*\ot \fsl(V)\ot W$, namely
\begin{align*}
&\langle x^3_1\ot y^2_1, x^3_2\ot y^2_1-x^3_1\ot y^1_1, x^3_1\ot y^1_1,x^2_1\ot y^2_1 \rangle,
\langle x^3_1\ot y^2_1, x^3_2\ot y^2_1-x^3_1\ot y^1_1, x^3_1\ot y^1_1,x^3_1\ot y^2_1 \rangle,
\\
&\langle x^3_1\ot y^2_1, x^3_2\ot y^2_1-x^3_1\ot y^1_1, x^2_1\ot y^2_1, x^3_1\ot y^2_1 \rangle,\\
&
\langle x^3_1\ot y^2_1, x^3_2\ot y^2_1-x^3_1\ot y^1_1, x^2_1\ot y^2_1, x^2_2\ot y^2_1-x^2_1\ot y^1_1  \rangle,
\langle x^3_1\ot y^2_1, x^3_2\ot y^2_1-x^3_1\ot y^1_1, x^3_1\ot y^2_2, x^3_2\ot y^2_2-x^3_1\ot y^1_2 \rangle,
\\
&
\langle x^3_1\ot y^2_1, x^3_2\ot y^2_1-x^3_1\ot y^1_1, x^2_1\ot y^2_1, x^1_1\ot y^2_1, \rangle,
\langle x^3_1\ot y^2_1, x^3_2\ot y^2_1-x^3_1\ot y^1_1, x^3_1\ot y^2_2, x^3_1\ot y^2_3, \rangle,\\
&\langle x^3_1\ot y^2_1,   x^2_1\ot y^2_1, x^1_1\ot y^2_1,x^3_1\ot y^2_2 \rangle,
\langle x^3_1\ot y^2_1, x^2_1\ot y^2_1 , x^3_1\ot y^2_2, x^3_1\ot y^2_3  \rangle.
\end{align*}
Here, to be a candidate $E_{110}'$, the map \eqref{210map} needs to have rank at most $11$.
Using the $U^*\leftrightarrow W$ symmetry, we see none of them pass the test:
they all have either \eqref{210map} or its $(120)$ version   of rank at least  $12$:
if we take $x^3_1\ot y^2_1$, two lowerings by $\fsl(W)$ and one by $\fsl(V)$ or $\fsl(U)$,
\eqref{210map}  has rank $14$. If we take two lowerings by $\fsl(V)$ and
one by $\fsl(W)$ \eqref{210map}   has rank $12$. By symmetry this covers all cases.\qed

For Theorem \ref{2nnbnds}(2), the result follows by an exhaustive computer search.

\section{Proofs of Theorems \ref{2nnbnds}(3) and \ref{mnnthm}}\label{outlinesect}

We   reduce the problem from upper-bounding  the kernels of the $(210)$ and $(120)$ maps
at arbitrary $\BB$-fixed spaces to a tractable computation in the following way.

Let $E_{110}' \subset U^* \ot \fsl(V) \ot W$ be a $\BB$-fixed subspace. Define
the {\it outer structure} of $E_{110}'$ to be the set of $\fsl(U)\oplus
\fsl(W)$ weights appearing in $E_{110}'$, counted with multiplicity. We identify the $\fsl(U)$ weights of $U^*$ and the
$\fsl(W)$ weights of $W$ each with $\{1,\ldots,\nnn\}$, where  $1$ corresponds to the
highest weight. In this way we consider the outer structure of $E_{110}'$ as a
subset of an $\nnn\times \nnn$ grid, with grid points possibly taken with
multiplicity.
 (For the purposes of the lemmas in section
\ref{twolems}, we view the grid as right justified to facilitate use of the
language of Young diagrams, but when viewed as weight diagrams, it is best
viewed as diamond shape with the $(1,1)$-vertex at the summit.)
We speak of the {\it inner structure} of $E_{110}'$ to be the particular weight
spaces which occur at each weight $(s,t) \in \nnn\times \nnn$. The set of  possible inner
structures over a grid point $(s,t)$ corresponds to the set of 
$\BB$-fixed   subspaces of   $\fsl(V)$.

We may filter $E_{110}'$ by $\BB$-fixed subspaces
such that  each quotient corresponds to the inner structure contribution over some site
$(s,t)$. Call such a filtration {\it admissible}. The kernel dimension of
either the $(210)$ or the $(120)$ map may be written as a telescoping sum of
differences of kernel dimensions corresponding to successive terms of such a
filtration. Thus, with respect to an admissible filtration, we may speak of the
contribution to the total kernel from the inner structure at grid point
$(s,t)$.  
We will bound the minimum of the kernel dimensions of the $(210)$ and
$(120)$ maps by first upper bounding the contribution to each kernel from the inner
structure at $(s,t)$ by a function of only $s$, $t$, and $j$, the inner structure
dimension. In particular, we obtain bounds independent of the outer structure
and the particular admissible filtration used to develop the dimension of the
total kernel. 
For $\fsl_2$, this is Lemma \ref{absl2} and for $\fsl_3$, this is
Lemma \ref{absl3}.
We then obtain bounds on the minimum of the kernel dimensions by
solving the resulting optimization problem over the possible outer structures
in Lemma \ref{alem2}.

\begin{lemma}\label{absl2}When $\tdim V=2$, $\tdim U=\tdim W=\nnn$, the differences in the dimensions of the kernels of the $(210)$-maps from a filtrand $\Sigma\subset E_{110}'$
such that the previous filtrand omits the   subspace at site $(s,t)$  and the dimension of the quotient
of the filtrands is $j$, equals the  function $a_js+b_j$ where 
\begin{center}\begin{tabular}{ccc}  
$j$ & $a_j$ &   $b_j$\\
\midrule
1 & $2$ &   $0$  \\
2 & $3$ &   $\nnn$   \\
3 & $4$ &  $2\nnn$.\\ 
\end{tabular}\end{center}
\end{lemma}

Lemma \ref{absl2} is proved in \S\ref{bvtwo}.

\begin{lemma}\label{absl3}
When $\tdim V=3$, $\tdim U=\tdim W=\nnn$, the  differences in the dimensions 
of the kernels of the $(210)$-maps from a filtrand $\Sigma\subset E_{110}'$
such that the previous filtrand omits the   subspace at site $(s,t)$  and the dimension of the quotient
of the filtrands is $j$, is bounded above by a function $a_js+b_j$ where 
\begin{center}
  \begin{tabular}{ccc}  
$j$ & $a_j$ &   $b_j$\\
\midrule
$1$ & $1$ &   $0$  \\
$2$ & $4$ &   $-1$   \\
$3$ & $10$ &  $-4$ \\ 
$4$ & $11$ &$-4$\\
\end{tabular}
\qquad
\begin{tabular}{ccc}  
$j$ & $a_j$ &   $b_j$\\
\midrule
$5$ & $15$ &$\nnn-4$\\
$6$ & $20$ &$ \nnn-6$\\
$7$ & $21$ &$ 2\nnn-6$\\
$8$ & $21$ &$ 3\nnn-6$.
\end{tabular}
\end{center}
\end{lemma}

 Lemma \ref{absl3} is proved in \S\ref{bvthree}.

In Lemma \ref{alem2} below the   linear functions of $s$  in the lemmas above appear as  $a_{\mu_{s,t}}s+b_{\mu_{s,t}}$.

For  a partition  $\lambda=(\lambda_1\hd \lambda_q)$, write $ \ell(\lambda)=q$ and  $n(\lambda) = \sum_i (i-1) \lambda_i$.
Let $\lambda'$ denote the conjugate partition.
Write $\mu$ for a Young tableau with integer labels. The
label in position $(s,t)$ is denoted $\mu_{s,t}$, and sums over $s,t$ are to be
taken over the boxes of $\mu$.

\begin{lemma}\label{alem2} 
Fix $k\in \BN$, $0 \le a_1 \le \cdots \le a_k$, and $b_i \in \RR$, $1\le i \le
  k$. Let $\mu$ be a Young tableau with labels in the set $\{1,\ldots,
  k\}$, non-increasing in rows and columns. Write $\rho = \sum_{s,t} \mu_{s,t}$. 
  Then
\be\label{bndlemsum}
    \min \Big\{ \sum_{s,t} a_{\mu_{s,t}}s + b_{\mu_{s,t}} ,
           \sum_{s,t} a_{\mu_{s,t}}t + b_{\mu_{s,t}} \Big\}
           \le \max_{1\le j\le k} \bigg\{ 
           \frac{a_j \rho^2}{8j^2} + (a_j+b_j) \frac{\rho}{j} \bigg\}.
\ene
\end{lemma}
Thus we obtain an upper bound on the minimum of the two kernels for {\it any} $E_{110}'$
in terms of the constants $a_j$ and $b_j$.

 Lemma \ref{alem2} is proved in \S\ref{twolems}.

We remark that when $\nnn$ is large,  if one takes 
the  outer structure     a balanced hook and the inner structure the same $j$ at each vertex,
assuming this is possible, 
one obtains the exact minimum of the kernels of
the $(210)$ and $(120)$ maps on the right hand side of \eqref{bndlemsum}
corresponding to $j=3$ in the $\fsl_2$ case and for the  $\fsl_3$ case, when
one takes  $j= 8$ (which is optimal for large $\nnn$)  one gets  within $2$ of the exact minimum.

\begin{proof}[Proof of Theorem \ref{2nnbnds}(3)]
  Let $E_{110}'$ be a $\BB$-fixed subspace and pick an admissible filtration.
  Following the discussion at the start of \S\ref{outlinesect}, 
  we apply Lemma \ref{alem2} with $k=3$, $\rho = \dim E_{110}'$, and $a_i$ and $b_i$ corresponding to 
  to the inner structure contributions obtained in Lemma \ref{absl2} to obtain
  an upper bound on the minimal kernel dimension of the $(120)$ and $(210)$
  maps. The resulting upper bound is $\max \{\frac{1}{4} \rho^2 + 2\rho,
    \frac{3}{32} \rho^2 + \frac{3+\nnn}{2} \rho , \frac{1}{18} \rho^2 +
  \frac{4+2\nnn}{3} \rho \}$. 

  Fix $\epsilon > 0$.
  We must show that if $\rho = (3\sqrt{6} - 6 - \epsilon) \nnn$, then each of 
  $\frac{1}{4} \rho^2 + 2\rho$, $\frac{3}{32} \rho^2 + \frac{3+\nnn}{2} \rho$,
  and $ \frac{1}{18} \rho^2 + \frac{4+2\nnn}{3} \rho $ is strictly smaller than
  $\nnn^2 + \rho$.
  Substituting and solving for $\nnn$, we obtain that this holds for the last
  expression when
  \[
    \nnn > \frac{6}{\epsilon} \frac{3\sqrt{6}+6 - \epsilon}{6\sqrt{6} - \epsilon},
  \]
  and when $\epsilon < \frac{1}{4} $, this condition implies the other two
  inequalities. 
\end{proof}

\begin{proof}[Proof of Theorem \ref{mnnthm}] 
  Proceeding in the same way as in the proof of Theorem \ref{2nnbnds}(3), we
  apply Lemma \ref{alem2} with $k=8$, $\rho = \dim E_{110}'$, and $a_i$ and
  $b_i$ corresponding to the inner structure contribution upper bounds obtained
  in Lemma \ref{absl3}. We obtain the smaller dimensional kernel between the
  $(120)$ and the $(210)$ maps is at most the largest of the following,
\begin{center}
\begin{tabular}{cc}  
  $j$ & Lemma \ref{alem2} \\ \midrule
$1$ & $ \frac 18\rho^2+\rho$                \\[2.5px]
$2$ & $\frac 18\rho^2+\frac 32\rho$         \\[2.5px]
$3$ & $\frac {5}{36}\rho^2+2\rho$           \\[2.5px]
$4$ & $\frac {11}{128}\rho^2+\frac{7}4\rho$ \\[2.5px]
\end{tabular}
\qquad
\begin{tabular}{cc}  
  $j$ & Lemma \ref{alem2} \\ \midrule 
$5$ & $\frac {3}{40}\rho^2+\frac{11+\nnn}5\rho$   \\[2.5px]
$6$ & $\frac {5}{72}\rho^2+ \frac{14+\nnn}6\rho$  \\[2.5px]
$7$ & $\frac {3}{56}\rho^2+\frac{15+2\nnn}7\rho$  \\[2.5px]
$8$ & $\frac {21}{512}\rho^2+\frac{15+3\nnn}8\rho$. \\[2.5px]
\end{tabular}
\end{center}
Parts (1) and (2) follow by direct calculation with the values above.

To prove part (4), fix $\epsilon > 0$.
  Substituting $\rho = (\frac{16}{21} \sqrt{78} - \frac{32}{7} - \epsilon) \nnn $ 
  and solving for $\nnn$ in the condition 
  $ \frac {21}{512}\rho^2+\frac{15+3\nnn}8\rho < \nnn^2 + \rho$, we obtain
  \[
    \nnn > \frac{64}{3\epsilon} \frac{16\sqrt{78} - 96 - 21\epsilon}{32\sqrt{78}
    - 21\epsilon}
  \]
  and when $\epsilon < \frac{1}{2} $, this condition implies the required
  inequality for the remaining seven terms, as required.  Part (3)
  follows as a special case of (4).
\end{proof}

\subsection{Proof of Lemma \ref{alem2}}\label{twolems}
We remark that the results in this section may be used for $M_{\langle \mmm\nnn\nnn\rangle}$ for
any $\nnn\geq\mmm$.

To establish Lemma \ref{alem2} we need two additional lemmas:

\begin{lemma} \label{lemma:singlebound}
  Let $\lambda$ be a partition not of the form $(n,2)$. Then $n(\lambda) \le
  \frac{1}{8}(\abs \lambda +\lambda_1' - \lambda_1)^2 - \frac{1}{8} $. In
  particular, for all $\lambda$, $n(\lambda) \le \frac{1}{8}(\abs \lambda
  +\lambda_1' - \lambda_1)^2$.
\end{lemma}
\begin{proof}
  We prove the result by induction on $\ell(\lambda')$. 
  When $\ell(\lambda') = 1$, we
  have $n(\lambda) = \binom{\lambda'_1}{2} = \frac{1}{2} (\lambda'_1 -
  \frac{1}{2} )^2 - \frac{1}{8} = \frac{1}{8} ( \abs{\lambda} + \lambda'_1 -
  \lambda_1)^2 - \frac{1}{8} $, 
  as required.
  Now, assume $k = \ell(\lambda') > 1$. Write $\mu$ for the partition where $\ell(\mu') =
  k - 1$ and $\mu'_i = \lambda'_i$, $i\le k-1$. 
  If $\lambda = (3,3)$, we are done by direct calculation, hence otherwise
  we may assume the result holds for $\mu$ by the induction hypothesis.
  \begin{align*}
    \textstyle n(\lambda) & \textstyle = n(\mu) + \binom{\lambda'_k}{2}  \\
    & \textstyle \le \frac{1}{8}(\abs \mu +\mu_1' - \mu_1)^2 - 
    \frac{1}{8}  + \binom{\lambda'_k}{2}  \\
    & \textstyle = \frac{1}{8}(\abs \lambda - \lambda'_k +\lambda_1' -
    (\lambda_1 - 1))^2 - \frac{1}{8} + \frac{1}{2} \lambda_k' (\lambda_k'-1) \\
    & \textstyle = \frac{1}{8}(\abs \lambda +\lambda_1' - \lambda_1)^2
    - \frac{1}{8} - \frac{1}{4} ( \abs \lambda +\lambda_1' - \lambda_1
    - \frac{5}{2}  \lambda'_k + \frac{1}{2} ) (\lambda'_k - 1)
\end{align*}
We must show the right hand term is non-positive. If $\lambda_k' = 1$, this is
  immediate; otherwise, we show the first factor is nonnegative.
  We have $\abs \lambda -\lambda_1 \ge k\lambda_k' - k$, so $ \abs \lambda
  +\lambda_1' - \lambda_1 - \frac{5}{2}  \lambda'_k + \frac{1}{2}  \ge 
  (\lambda_1' - \lambda_k') + 
  \frac{2k-3}{2} (\lambda_k'-1) - 1 $. 
  If $k =2$, then by assumption $\lambda_1' \ge 3$, and considering separately the
  cases $\lambda_2' = 2$ and $\lambda_2' \ge 3$ yields that the first factor is
  nonnegative.
  Otherwise $k\ge 3$, and because $\lambda_k' \ge 2$, the first factor is
  nonnegative. This completes the proof.
\end{proof}

\begin{lemma}\label{lemma:opt}
  Fix $k\in \BN$, 
  $c_i \ge 0$, $d_i \in \RR$, for $1\le i\le k$. Write $C_j = \sum_{i=1}^j c_i$
  and $D_j = \sum_{i=1}^j d_i$. For all choices of $x_i,y_j$ satisfying the constraints
  $x_1 \ge  \cdots \ge x_k \ge 0$, $y_1 \ge  \cdots \ge y_k \ge 0$,  and $\sum_i x_i + y_i =
  \rho$, the following inequality holds:
 \be\label{braceterm}
    \min \Big\{ \sum_{i\le k} c_ix_i^2 + d_i(x_i+y_i), 
        \sum_{i\le k} c_iy_i^2 + d_i(x_i+y_i)
    \Big\} \le 
  \max_{1\le j\le k} \bigg\{ \frac{\rho^2}{4j^2} C_j + \frac{\rho}{j} D_j \bigg\}.
 \ene
\end{lemma}

\begin{remark} 
The maximum is achieved when $x_1=\cdots =x_j=y_1=\cdots =y_j=\frac{\rho}{2j}$ and $x_s,y_s=0$ for $s>j$, for some $j$.
\end{remark}

\begin{proof}
  As both the left and right hand sides are continuous in the $c_i$, it suffices
  to prove the result under the assumption $c_i > 0$.
  The idea of the proof is the following: any choice of $x_i$ and $y_i$ which
  has at least two degrees of freedom inside its defining polytope can be perturbed
  in such a way that the local linear approximations to the two polynomials on the left
  hand side do not decrease; that is, two closed half planes in $\RR^2$
  containing $(0,0)$ also intersect aside from $(0,0)$. Each polynomial on the
  left strictly exceeds its linear approximation at any point, and thus  one can
  strictly improve the left hand side with a perturbation. The
  case of at most one degree of freedom  is settled directly.
  
%The idea of the proof is as follows: first, if   three or more of the $x_i',y_j'$ are nonzero, we may preturb three
%of them, by say $\ep_1,\ep_2,\ep_3$. Here we are subject to the constraint that $\ep_1+\ep_2+\ep_3=0$ to preserve
%$\sum_i x_i' + y_i' =\rho$.   We impose a second constraint on the $\ep_i$
%to keep the larger of the two summations constant. We still have one degree of freedom which we use to increase the value of the 
%smaller summand.
%In this way we reduce to the case of one $x_i'$ and one $y_j'$ being nonzero. 
%We  prove  this case by direct  optimization.
%We observe that if they are both
%the same and equal to $j$, we get the term in braces on the right hand side of  \eqref{braceterm}.

Write $x_{k+1}= y_{k+1} =0$, and define $x_i' = x_i- x_{i+1}$ and $y_i' =
  y_i-y_{i+1}$ so that $x_i = \sum_{j=i}^k x_j'$ and $y_i = \sum_{j=i}^k y_j'$.
  Then $x_i', y_i' \ge 0$ and $\sum_{i=1}^k i (x_i' + y_i') = \rho$.
Suppose at least three of the $ x_i' $,
  $ y_j' $ are nonzero, we will show  the expression on the left hand side of \eqref{braceterm}
is not maximal. Write three of the nonzero $ x_i',y_j'$   as $\ol x,\ol y,\ol z$.
Replace them by $\ol x+\ep_1$, $\ol y+\ep_2$, $\ol z+\ep_3$, with the $\ep_i$ to be
determined. This will preserve the summation to $\rho$ only
if $\ep_1+\ep_2+\ep_3=0$, so we require this. Substitute  these values   into 
  $E_L:=\sum_{i\le k} c_ix_i^2 + d_i(x_i+y_i)$ and $E_R:=\sum_{i\le k} c_iy_i^2 +
  d_i(x_i+y_i)$.  View $E_L,E_R$  as two polynomial expressions in the $\ep_j$.
  Then $E_L=\sum_i c_i S_{L,i}^2 + L_L + d$, $E_R=\sum_i c_i S_{R,i}^2 + L_r + d$ where
  $S_{L,i},S_{R,i}$ and $L_L,L_R$ are linear forms in the $\ep_i$, and $d\in \RR$. Each
  $S_{L,i},S_{R,i}$ is a sum of some subset of the $\ep_i$, and  
  the union of them  span $\langle \ep_1,\ep_2,\ep_3\rangle/\langle \sum\ep_j=0 \rangle$.
  Consider the linear map $T = L_L \oplus L_R
   : \langle \ep_1,\ep_2,\ep_3\rangle/\langle \sum\ep_j=0 \rangle \to \RR^2$. If $T$ is nonsingular,
  then for any $\epsilon>0$, there are constants  $\ol\ep_j$, with $\sum\ol\ep_j=0$ so that $T(\ol\ep_1,\ol\ep_2,\ol\ep_3) =
  (\epsilon,\epsilon)$, and it is possible to choose $\epsilon$ so that
  $ \ol x  + \ol\ep_1,  \ol y  + \ol\ep_2,  \ol z  +\ol\ep_3 \ge 0$. Then this new
  assignment strictly improves the old one. Otherwise, if $T$ is singular, then
  there is an  admissible  $(\ol\ep_1,\ol\ep_2,\ol\ep_3) \ne 0$ in the kernel of $T$, where again we may
  assume the the same non-negativity condition. The
  corresponding assignment does not change $L_L,L_R$, but
  as the $S_{L,i},S_{R,i}$ span the linear forms, at least one them is nonzero. 
  Consequently, at least one of the modified   
  $E_L,E_R$
  is strictly larger after the perturbation, and neither is smaller. If, say,
  only $E_L$ is strictly larger,
  and $ x_i'  > 0$, we may substitute $ x_i'  - \epsilon$ and
  $ y_i'  + \epsilon$ for $ x_i' $ and $ y_i' $ for some
  $\epsilon >0$ to make both $E_L$ and $E_R$ strictly larger.

  Thus, the left hand side is maximized at an assignment where at most two
  of $x_i'$ and $y_i'$ are nonzero. It is clear that at least one
  of each of $x_i'$ and $y_i'$ must be nonzero, so there is
  exactly one of each, say $x_s' = \alpha$ and $y_t' = \beta$.
  It is clear at the maximum that
  $ \sum_{i\le k} c_ix_i^2 + d_i(x_i+y_i) =  \sum_{i\le k} c_iy_i^2 +
  d_i(x_i+y_i) $,
  from which it follows that
  $ \alpha^2 C_s = \sum_{i\le k} c_ix_i^2 =\sum_{i\le k} c_iy_i^2 
    = \beta^2 C_t $ and $\alpha \sqrt{C_s} = \beta \sqrt{C_t}$. We also have
    $s\alpha + t\beta = \rho$. Notice that 
    \[
      \alpha = \frac{\rho \sqrt{C_t}}{s\sqrt{C_t} + t\sqrt{C_s}},\quad
      \beta = \frac{\rho \sqrt{C_s}}{s\sqrt{C_t} + t\sqrt{C_s}}
    \]
    satisfy the equations, so that the optimal value obtained is
  \[
    \sum_{i\le k} c_ix_i^2 + d_i(x_i+y_i)  = 
    \alpha^2 C_s
    + \alpha D_s
    + \beta D_t
    = \frac{\rho}{s\sqrt{C_t} + t\sqrt{C_s}}
\left( \frac{\rho C_s C_t}{s\sqrt{C_t} + t\sqrt{C_s}} + \sqrt{C_t}D_s + \sqrt{C_s} D_t \right)
  \]
  By the arithmetic mean-harmonic mean inequality, we have
  \[
    \frac{ \rho C_s C_t}{s\sqrt{C_t} + t\sqrt{C_s}}
    = \frac{ \rho}{ \frac{s}{C_s \sqrt{C_t}} +  \frac{t}{C_t \sqrt{C_s}}}
    \le 
     \frac{ \rho}{4} \bigg[ \frac{C_s \sqrt{C_t}}{s} +  \frac{C_t \sqrt{C_s}}{t}
     \bigg],
    \]
    so that
  \begin{align*}
    \frac{\rho C_s C_t}{s\sqrt{C_t} + t\sqrt{C_s}} + \sqrt{C_t}D_s + \sqrt{C_s} D_t 
    &\le 
     \frac{ \rho}{4} \bigg[ \frac{C_s \sqrt{C_t}}{s} +  \frac{C_t \sqrt{C_s}}{t}
     \bigg] +
    \sqrt{C_t}D_s + \sqrt{C_s} D_t  \\
    &=
    \frac{s\sqrt{C_t}+t\sqrt{C_s}}{\rho} 
    \left[\frac{s\alpha} {\rho} \Big(
    \frac{\rho^2}{4s^2} C_s + \frac{\rho}{s} D_s
    \Big) +
    \frac{t\beta} {\rho}\Big(
    \frac{\rho^2}{4t^2} C_t + \frac{\rho}{t} D_t
    \Big) 
    \right] \\
    &\le
    \frac{s\sqrt{C_t}+t\sqrt{C_s}}{\rho} 
    \max \bigg\{ 
    \frac{\rho^2}{4s^2} C_s + \frac{\rho}{s} D_s, 
    \frac{\rho^2}{4t^2} C_t + \frac{\rho}{t} D_t
    \Big\},
  \end{align*}
    with the last inequality from the fact that $ \frac{s\alpha}{\rho} +
    \frac{t\beta}{\rho} = 1$. Multiplying both sides by 
    $ \frac{\rho}{s\sqrt{C_t} + t\sqrt{C_s}} $, we conclude the optimal value
    is achieved at one of the claimed values.
\end{proof}

\begin{proof}[Proof of Lemma \ref{alem2}]
  For each $1\le i\le k$, let $\lambda^i$ be the partition corresponding to
  the boxes of $\mu$ labeled $\ge i$
  Write $a_0 = b_0 = 0$,  Then,
  \begin{align*}
    \textstyle \sum_{s,t} a_{\mu_{s,t}}s + b_{\mu_{s,t}} 
    &= \textstyle
    \sum_{s,t} \sum_{i=1}^{\mu_{s,t}} (a_i-a_{i-1})s + b_i-b_{i-1} \\
    &= \textstyle \sum_{i=1}^k \sum_{s,t \in \lambda^i} (a_i-a_{i-1})s + b_i-b_{i-1} \\
    % &= \sum_{i=1}^k \sum_{t \le \lambda^i_s} (a_i-a_{i-1})s + b_i-b_{i-1} \\
    &= \textstyle \sum_{i=1}^k (a_i - a_{i-1}) n(\lambda^i) + (a_i - a_{i-1} + b_i -
    b_{i-1})\abs{\lambda^i} \\
    &= \textstyle \sum_{i=1}^k (a_i - a_{i-1}) n(\lambda^i) + (a_i - a_{i-1} + b_i -
    b_{i-1})\abs{\lambda^i} \\
    &\le \textstyle \sum_{i=1}^k \left[\frac{1}{2}  (a_i-a_{i-1})\right] 
    \left(\frac{1}{2} (\abs{\lambda^i} + (\lambda^i)'_1- \lambda^i_1)\right)^2 + \left[a_i - a_{i-1} + b_i - b_{i-1}\right]\abs{\lambda^i}
  \end{align*}
   where we have used Lemma
  \ref{lemma:singlebound}  to obtain the last inequality.
  Set 
  \begin{align*}
  c_i &= \textstyle\frac{1}{2}  (a_i-a_{i-1})\\
  d_i &=\textstyle a_i - a_{i-1} + b_i - b_{i-1}\\
    x_i &=\textstyle \frac{1}{2} (\abs{\lambda^i} + (\lambda^i)'_1- \lambda^i_1)\\
   y_i &=\textstyle \frac{1}{2} (\abs{\lambda^i} - (\lambda^i)'_1+ \lambda^i_1).
   \end{align*}
Then the last line becomes
 $$ \sum_{i=1}^k c_i x_i^2 + d_i (x_i + y_i).
 $$
  Similarly,  $ \sum_{s,t} a_{\mu_{s,t}}t + b_{\mu_{s,t}} 
  \le \sum_{i=1}^k c_i y_i^2 + d_i (x_i + y_i) $. Now, $\sum_i x_i +
  y_i = \sum_i \abs{\lambda^i} = \rho$ and the $x_i$ and $y_i$ are each
  nonnegative and non-increasing.  Hence, by Lemma \ref{lemma:opt},
  \[
    \min \Big\{ \sum_{s,t} a_{\mu_{s,t}}s +
  b_{\mu_{s,t}}, \sum_{s,t} a_{\mu_{s,t}}t + b_{\mu_{s,t}} \Big\}
  = \max_{1\le j\le k} \bigg\{ \frac{a_j \rho^2}{8j^2} + (a_j+b_j)
  \frac{\rho}{j} \bigg\},
  \]
  as required.
\end{proof}

%\begin{remark} The bound given by  Lemma \ref{alem2} for the minimum dimension of the kernels  is tight in the cases   where we
%use the exact dimensions of the new kernels.
%\end{remark}

\subsection{Additional information about the $(210)$ and $(120)$ maps for $M_{\langle\mmm\nnn\nnn\rangle}$}

\begin{proposition}\label{filterprop} Give $E_{110}'\subset U^*\ot \fsl(V)\ot W$   an admissible filtration.
Let $\Sigma_q\subset E_{110}'$ be a filtrand, and let $(s,t)$ be the   grid vertex of $\Sigma_q\backslash
\Sigma_{q-1}$.
Write $\Sigma_q\backslash \Sigma_{q-1}=u^{\nnn-s+1}\ot X\ot w_t$ for some $\BB$-fixed subspace $X\subseteq \fsl(V)$.
Then the dimension of the difference of the  kernels of the 
  $(210)$   maps   for $\Sigma_q$ and $\Sigma_{q-1}$ 
  is  of the form $as+b$ where
$a,b$ depend only on $X$. Similarly the dimension of the difference of the kernels of the $(120)$ maps is of the
form $at+b$, with the same $a,b$.
\end{proposition}
\begin{proof}
By symmetry, it is sufficient to consider the $(210)$ map.
Write $U^{*(s)}=\langle u^{\nnn},u^{\nnn-1}\hd u^{\nnn-s+1}\rangle$ and
$W_{(t)}=\langle w_1\hd w_t\rangle$.
Consider the image of $A\ot u^{\nnn-s+1}\ot X\ot w_t$ under the map \eqref{210map}.
Since $E_{110}'$ is Borel fixed,  
   the images of 
$U^{*(s-1)}\ot (V\ot X)\ot W_t$ and $U^{*(s)}\ot (V\ot X)\ot W_{t-1}$
under \eqref{210map} are in the image of \eqref{210map}
applied to $\Sigma_{q-1}$. 
%Thus any new contribution to  the kernel must be an element of 
%$U^*\ot u^{\nnn-s+1}\ot(V\ot  X)\ot  w_t+ u^{\nnn-s+1}\ot U^{*(s)}\ot(V\ot X)\ot w_t$ where the
%first term is nonzero.
%More precisely, 
The new contribution to the kernel is thus
\begin{align}
\nonumber&U^{*(s)}\cdot u^{\nnn-s+1}\ot [(V\ot X)\cap V_{2\o_1+\o_{\bv-1}}]\ot w_t\\
&\label{Xker}
\op U^{*(s-1)}\ww u^{\nnn-s+1}\ot [(V\ot X)\cap V_{ \o_2+\o_{\bv-1}}] \ot w_t
\\
&\nonumber
\op
U^*\ot u^{\nnn-s+1}\ot [(V\ot X)\cap V_{\o_1}]\ot w_t.
\end{align}
This has dimension 
$$s\tdim [(V\ot X)\cap   V_{2\o_1+\o_{\bv-1}}]+
(s-1) \tdim [(V\ot X)\cap    V_{\o_2+\o_{\bv-1}}]  
+\nnn \tdim [(V\ot X)\cap  V_{\o_1}].
$$
We conclude that  
\begin{align}\label{abcount}
a&=\tdim [(V\ot X)\cap    V_{2\o_1+\o_{\bv-1}}]+
  \tdim [(V\ot X)\cap    V_{\o_2+\o_{\bv-1}}]\\
b&\nonumber =\nnn\tdim [(V\ot X)\cap   V_{\o_1}]-\tdim [(V\ot X)\cap  V_{\o_2+\o_{\bv-1}}].
\end{align}
\end{proof}

 \begin{example}\label{sqrt2ex}
Let $\mmm>2$ and let $\t\leq \binom{\mmm}2$. Consider
$$
E_{110}'=
\langle u^\nnn\hd u^{\nnn-\s+1}\rangle \ot \langle v_1\hd v_\t\rangle\ot  v^\mmm  \ot   w_1 +
  u^\nnn \ot   v_1  \ot \langle v^\mmm\hd v^{\mmm-\t+1}\rangle \ot \langle w_1\hd w_\s\rangle .
$$
Since $S^2V\ot V^*= V_{2\o_1+\o_{\mmm-1}}\op V$ and $\La 2 V\ot V^*= V_{ \o_2+\o_{\mmm-1}}\op V$,
 this space has the following contributions to the kernel coming from the left side of the grid:
$$
S^2\langle u^\nnn\hd u^{\nnn-\s+1}\rangle\ot S^2\langle v_1\hd v_\t\rangle\ot   v^\mmm  \ot   w_1 
$$
and
$$
\La 2\langle u^\nnn\hd u^{\nnn-\s+1}\rangle\ot \La 2\langle v_1\hd v_\t\rangle\ot   v^\mmm  \ot   w_1  .
$$
Geometrically,   $S^2(\BC^{\s\t})=S^2(\BC^\s\ot \BC^\t)=S^2\BC^\s\ot S^2\BC^\t \op \La 2 \BC^\s\ot \La 2\BC^\t$,
so the contribution has dimension  $\binom{\s\t+1}2$.
 
The right side of the grid just contributes  
$$
(u^\nnn)^2\ot v_1^2\ot   \langle v^\mmm\hd v^{\mmm-\t+1}\rangle \ot
  \langle w_1\hd w_\s\rangle
$$
which has dimension $\s\t$. Since $\s\t=\frac 12\tdim E_{110}'$, we see the total only depends on $\s\t$.
Note that the dimension is also independent of $\tdim V$.
\end{example}

Example \ref{sqrt2ex}  implies:

\begin{proposition} 
 For any   $\mmm$, $\nnn$,   
one cannot prove better than $\ur(M_{\langle\mmm\nnn\nnn\rangle})\geq \nnn^2+(2\sqrt{2}-\ep)\nnn$
for all   large $\nnn$ given any small $\ep$
just using the $(210)$ and $(120)$ tests.
\end{proposition}
\begin{proof} 
Using the above $E_{110}'$ with $\s=\frac 12\tdim(E_{110}')$ and $\t=1$,
we would need
$
\binom{\s+1}2<\nnn^2+2\s 
$, i.e., $\frac{\s^2}2-\frac{3\s}2<\nnn^2$
to pass the $(210)$-test.
\end{proof}

In future work we plan to   geometrically analyze the $(111)$ test for matrix multiplication.

\subsection{Proof of Lemma \ref{absl2}} \label{bvtwo}
We compute, for each $1\leq j\leq 3$, the values $a_j$ and $b_j$ such that   the additional contribution
to the kernel of the $(210)$   map  obtained by adding a $j$-dimensional $\BB$-closed subset of
$\fsl_2$ at grid site $(s,t)$ is    $a_js+b_j$. 

Here  $V_{\o_2+\o_{\bv-1}}$ is zero and $V_{2\o_1+\o_{\bv-1}}$ is  
$S^3V$ (which has dimension $4$), and thus the target  $\La 2U^*\ot S^{3}V\ot W$ is multiplicity free.
Moreover, any monomial $x^i_s\ot x^k_t\ot y^u_j$ having the same
weight as an element of $\La 2U^*\ot S^{3 }V\ot W$ must project
nonzero onto it  because the   basis vectors of both the symmetric and exterior powers of the standard representation
are linear combinations of tensor monomials of the same weight with all coefficients nonzero.
Thus to compute the image of the map
\eqref{210map}, we just need to keep track of the   weights.

Since $V_{\o_2+\o_{\bv-1}}$ is zero, \eqref{Xker} just has two summands.

Case 1: $X=\langle v_1\ot v^2\rangle$. Here
  $(V\ot  v_1\ot v^2)\cap S^3V$ is two dimensional  as the weight $-3$ does not appear in $V\ot  v_1\ot v^2$
and by Remark \ref{zerorem} there can be no intersection with $V$. We conclude $a_1=2$, $b_1=0$.

Case 2: $X=\langle v_1\ot v^2, v_1\ot v^1-v_2\ot v^2\rangle$. Here 
$(V\ot  X)\cap S^3V$ is three  dimensional  as all weights of $S^3V$ appear in $V\ot X$
and by Remark \ref{zerorem} there is a one-dimensional  intersection with $V$. We conclude $a_2=3$, $b_2=\nnn$.
 
 Case 3: $X=\fsl(V)$. Here
$(V\ot  \fsl(V))\cap S^3V$ is three  dimensional   
and by Remark \ref{zerorem} there is a two-dimensional  intersection with $V$. We conclude $a_3=3$, $b_3=2\nnn$.
 \qed

\subsection{Proof of   Lemma \ref{absl3} }\label{bvthree}
We bound, for each $1\leq j\leq 8$, the values $a_j$ and $b_j$ such that  the additional contribution
to the kernel of the $(210)$   map  obtained by adding a $j$-dimensional $\BB$-closed subset of
$\fsl_3$ at grid site $(s,t)$ is    at most  $a_js+b_j$. 

Here $V_{2\o_1+\o_{\bv-1}}=V_{2\o_1+\o_2}$, which has dimension $15$, and
$V_{\o_2+\o_{\bv-1}}=V_{2\o_2}$,  which has dimension $6$.

In the weight diagram for the adjoint
representation of $\fsl_3=\fsl(V)$, there are five levels from the highest weight to the lowest. The first and last level
have one element and the others have two each.
The third level is the weight zero subspace.  
By Remark \ref{zerorem}, there can only be a contribution to the kernel from $V_{\o_1}$ when a level is filled.
More precisely, the highest weight vector of $V_{\o_1}$ can only appear when level $3$ is filled, the next weight  
vector can only appear when level $4$ is filled, and the full module can only appear when all $\fsl(V)$ is used.

Write $a\o_1+b\o_2= [a,b]$ for weights. We have the following weight space decompositions, where the number
outside the brackets is the multiplicity:
\begin{align*}
V_{\o_1}\ \ \ \ \  & 1 [ 1, 0] +1 [-1, 1] +1 [ 0,-1]\\
V_{2\o_2}\ \ \ \ \  &1 [ 0, 2] +1 [ 1, 0] +1 [ 2,-2] +1 [-1, 1] +1 [ 0,-1] +1 [-2, 0]\\
V_{2\o_1+\o_2}\ \ \ \ \  & 1 [ 2, 1] +1 [ 3,-1] +1 [ 0, 2] +2 [ 1, 0] +1 [-2, 3] +1 [ 2,-2]\\
\ \ \ \ \  \ \ &
+
2 [-1, 1] +2 [ 0,-1] +1 [-3, 2] +1 [ 1,-3] +1 [-2, 0] +1 [-1,-2].
\end{align*}
We compare these with the weights of the following spaces:
\begin{align*}
V\ot v_1\ot v^3 \ \ \ \ \  & [2,1]+[0,2]+[1,0]\\
V\ot v_2\ot v^3 \ \ \ \ \  &  [0,2]+[-2,3]+[-1,1]\\
V\ot v_1\ot v^2 \ \ \ \ \  & [3,-1]+[ 1,-1]+[2,-2]\\
V\ot wt(0) \ \ \ \ \   & [1,0]+[-1,1]+[0,-1]\\
V\ot v_3\ot v^2  \ \ \ \ \  & [2,-2]+ [0,-1]+[1,-3]\\
V\ot v_2\ot v^1 \ \ \ \ \  & [-1,1]+[-2,1]+[-2,0].
\end{align*}
We expect that a vector only is in the kernel when it is forced to be so for multiplicity reasons, but
we only proved this   for the module $V_{\o_1}$ (Remark \ref{zerorem}). For some of the other modules, in what follows we upper bound the dimension of the kernel
by assuming anything that could potentially be in the kernel is there. For the cases of $j=1$, $j=2$, $j=5$ and $j=8$ 
  our computation of   the maximum kernel is exact.
  We repeatedly utilize \eqref{abcount}.

\subsubsection*{Case $j=1$} Here $X=\langle v_1\ot v^3\rangle$. By direct computation
(or see Example \ref{sqrt2ex})
  $(V\ot X)\cap V_{2\o_1+\o_2}=\langle v_1\ot v_1\ot v^3\rangle$,   by comparing weights,
 $(V\ot X)\cap V_{2\o_2}=0$, and by Remark \ref{zerorem} $(V\ot V)\cap V_{\o_1}=0$.
 We conclude $a_1=1$, $b_1=0$.
 
 \subsubsection*{Case $j=2$} 
 There are two possibilities:  $X=  \langle  v_1,v_2\rangle\ot v^3$
 or $v_1\ot  \langle v^2, v^3\rangle$.
 In the first case,  
  $(V\ot X)\cap V_{2\o_1+\o_2}=\langle (v_1)^2 v_1v_2, (v_2)^2\rangle\ot v^3$,  
 and 
 $(V\ot X)\cap V_{2\o_2}=\langle   v_1\ww v_2 \rangle\ot v^3$.
 In the second case,  
  $(V\ot X)\cap V_{2\o_1+\o_2}=  (v_1)^2   \ot \langle v^3,v^2\rangle$,  
 and 
 $(V\ot X)\cap V_{2\o_2}=0$.
 Since the first gives larger values, we obtain $a_2=4$, $b_2=-1$.
 
 \subsubsection*{Case $j=3$} 
 There are three possibilities:  
 \begin{align*}X&=\langle v_1\ot v^3, v_1\ot v^2,  v_2\ot v^3\rangle,\\
  X&=\langle v_1\ot v^3, v_1\ot v^2,  (v_1\ot v^1-v_3\ot v^3)\rangle,\\
 X&=\langle v_1\ot v^3, v_1\ot v^2,  (v_2\ot v^2-v_3\ot v^3)\rangle.
 \end{align*}
 The first satisfies (see Example \ref{sqrt2ex}) $\tdim [(V\ot X)\cap V_{2\o_1+\o_2}]= 4$
 and $\tdim[(V\ot X)\cap V_{2\o_2}]= 1$, which would give $a_3=5$, $b_3=-1$.
 The third of these has the largest upper bound,  namely $a_3=10$, $b_3=-4$.
 Here $10=4+6$ with the $4$ being exact from the previous case
 and the $6=3+3$ estimated, as 
   all three new vectors over the previous case from $V\ot (v_2\ot v^2-v_3\ot v^3)$ could potentially occur in
both the modules $S^2U^*\ot V_{2\o_1+\o_2}$ and $\La 2 U^*\ot V_{2\o_2}$.

\subsubsection*{Case $j=4$}
Here $X$ must consist of  $\langle v_1\ot v^3, v_1\ot v^2,  v_2\ot v^3\rangle$ plus a weight zero vector.
Again, the weight zero vector could potentially contribute as above, so we get
the estimate $a_4=11=5+6$ and $b_4=-4=-1-3$ where the first terms in the summations  are from 
$\langle v_1\ot v^3, v_1\ot v^2,  v_2\ot v^3\rangle$ and the second from
estimating the contribution of  the weight zero vector.
 
\subsubsection*{Case $j=5$}
Here  there are three possibilities for $X$. Th largest kernels occur
when the span of all weight spaces greater than or equal to zero. We compute the exact kernel:
Since we do not include
the last two  levels in $\fsl(V)$, 
the weights $[2,-2]$ and $[-2,0]$, which appear in $V_{2\o_2}$,  do not
contribute the kernel, so we obtain a $6-2=4$ dimensional intersection.  
The weights $[-3,2],[-2,0],[- 1,-2],[1,-3]$, which appear in $V_{2\o_1+\o_2}$, 
do not contribute to the kernel, so we obtain a $15-4=11$ dimensional intersection,
and   only the  highest weight vector from $V_{\o_1}$  contributes to
the kernel by Remark \ref{zerorem}  so 
we obtain   $a_5=15=4+11$ and $b_5=\nnn-4$.

\subsubsection*{Case $j=6$}
Here there are two choices, the higher upper-bound occurs when we add $v_3\ot v^2$
to the previous case, which
could potentially contribute a three dimensional intersection with 
$V_{2\o_1+\o_2}$ and a two dimensional intersection with $V_{2\o_2}$
   so we obtain the upper bounds
$a_6=20$ and $b_6=\nnn-6$.

\subsubsection*{Case $j=7$}
Here there is a unique choice, just omitting the lowest weight vector.
Here we use the calculation for the $j=8$ case and subtract $\nnn$ by Remark \ref{zerorem}
to get the upper bound $a_7=21$, $b_7=2\nnn-6$.

\subsubsection*{Case $j=8$}
Here we take the full $\fsl(V)$. Using \eqref{vslv} we obtain
  $a_7=21$, $b_7=3\nnn-6$.
 \qed

\section{Proof that $\ur(M_{\langle \lll,\mmm,\nnn\rangle})\geq 
\ur(M_{\langle \lll-1,\mmm,\nnn\rangle})+1$.}\label{Lickappen}
Here is a simple proof of the statement, which was originally shown in  \cite{MR86c:68040}.
By the border substitution method \cite{MR3633766}, for any tensor $T\in A\ot B\ot C$
$$
\ur(T)\geq \tmin_{A'\subset A^*}\ur(T|_{A'\ot B^*\ot C^*})+1,
$$
where $A'\subset A^*$ is
a hyperplane. Moreover, if $T$ has   
symmetry group $G_T$, and $G_T$ has a unique closed orbit in $\BP A^*$,
then we may restrict $A'$ to be a point of that closed orbit. In the case of matrix multiplication,
$G_{M_{\langle \lll,\mmm,\nnn\rangle}}\supset SL(U)\times SL(V)$ can degenerate  any point in $\BP A=\BP(U^*\ot V) $
to the annihilator of $x^1_\lll$, so it amounts to taking $T|_{A'\ot B^*\ot C^*}$ to be the reduced matrix multiplication
tensor with $x^1_\lll=0$. But now we may (using $GL(A)\times GL(B)\times GL(C)$) degenerate this tensor
further to set all $x^i_\lll$ and $y^\lll_j$ to zero to obtain the result.

\bibliographystyle{amsplain}

\bibliography{Lmatrix}

\end{document}